\documentclass[a4paper,11pt]{amsart}

\usepackage[utf8]{inputenc}
\usepackage[T1]{fontenc}

\usepackage[margin=1in]{geometry}
\usepackage[depth=2]{bookmark}

\usepackage{amsmath}
\usepackage{amssymb}
\usepackage{amsthm}

\usepackage{eucal}
\usepackage{enumitem}
\usepackage{tikz-cd}
\usepackage{xcolor}
\usepackage{xstring}

\theoremstyle{plain}
\newtheorem{thm}{Theorem}[section]
\newtheorem{prop}[thm]{Proposition}
\newtheorem{cor}[thm]{Corollary}
\newtheorem{lemma}[thm]{Lemma}

\theoremstyle{definition}
\newtheorem{dfn}[thm]{Definition}
\newtheorem{rmk}[thm]{Remark}
\newtheorem{ex}[thm]{Example}
\newtheorem{question}[thm]{Question}
\newcommand{\Acal}{\mathcal{A}} \newcommand{\Bcal}{\mathcal{B}} \newcommand{\Ccal}{\mathcal{C}}
  \newcommand{\Fcal}{\mathcal{F}}
\newcommand{\Gcal}{\mathcal{G}} \newcommand{\Hcal}{\mathcal{H}}

\newcommand{\Pcal}{\mathcal{P}}  
\newcommand{\Scal}{\mathcal{S}} \newcommand{\Tcal}{\mathcal{T}} \newcommand{\Ucal}{\mathcal{U}}
\newcommand{\Vcal}{\mathcal{V}}  \newcommand{\Xcal}{\mathcal{X}}
\newcommand{\Ycal}{\mathcal{Y}} \newcommand{\Zcal}{\mathcal{Z}}

\newcommand{\Abb}{\mathbb{A}}  
\newcommand{\Dbb}{\mathbb{D}}

\newcommand{\Pbb}{\mathbb{P}}

 \newcommand{\Zbb}{\mathbb{Z}}

\renewcommand{\phi}{\varphi}
\renewcommand{\epsilon}{\varepsilon}
\renewcommand{\theta}{\vartheta}
\DeclareMathOperator{\gldim}{gldim}
\DeclareMathOperator{\Hom}{\mathsf{Hom}}
\DeclareMathOperator{\End}{\mathsf{End}}
\DeclareMathOperator{\Proj}{\mathsf{Proj}}
\DeclareMathOperator{\proj}{\mathsf{proj}}
\DeclareMathOperator{\inj}{\mathsf{inj}}
\DeclareMathOperator{\Inj}{\mathsf{Inj}}
\DeclareMathOperator{\Add}{\mathsf{Add}}

\DeclareMathOperator{\add}{\mathsf{add}}

\DeclareMathOperator{\gen}{\mathsf{gen}}
\DeclareMathOperator{\cogen}{\mathsf{cogen}}
\DeclareMathOperator{\tr}{\mathsf{tr}}
\DeclareMathOperator{\rej}{\mathsf{rej}}
\DeclareMathOperator{\rad}{\mathsf{rad}}
\DeclareMathOperator{\sub}{\mathsf{cogen}}
\DeclareMathOperator{\fac}{\mathsf{gen}}

\DeclareMathOperator{\simp}{\mathsf{simp}}
\DeclareMathOperator{\soc}{\mathsf{soc}}
\let\top\relax
\DeclareMathOperator{\top}{\mathsf{top}}
\DeclareMathOperator{\filt}{\mathsf{filt}}
\DeclareMathOperator{\Ext}{\mathsf{Ext}}

\DeclareMathOperator{\ftors}{\mathsf{f-tors}}
\DeclareMathOperator{\tors}{\mathsf{tors}}
\DeclareMathOperator{\D}{\mathsf{D}}
\DeclareMathOperator{\K}{\mathsf{K}}
\DeclareMathOperator{\Mod}{\mathsf{Mod}}
\let\mod\relax
\DeclareMathOperator{\mod}{\mathsf{mod}}
\DeclareMathOperator{\Spec}{\mathsf{Spec}}

\DeclareMathOperator{\Ass}{\mathsf{Ass}}
\DeclareMathOperator{\supp}{\mathsf{supp}}

\newcommand{\pf}{\mathfrak{p}}

\newcommand{\dfour}[4]{\begin{smallmatrix}#1#2\!\begin{smallmatrix}#3\\#4\end{smallmatrix}\end{smallmatrix}}
\begin{document}

\title{Detecting derived equivalences with the CHZ criterion}
\author{Sergio Pavon}
\address[Sergio Pavon]{Dipartimento di Informatica, Università degli Studi di
Verona, Strada le Grazie 15 37134 Verona, Italy.}
\email{sergio.pavon@univr.it}

\begin{abstract}
	In 2018, Chen, Han and Zhou introduced a criterion to determine whether the
	HRS-tilt at a given torsion pair induces derived equivalence. We showcase four
	applications of this criterion: to injective- or projective-splitting torsion
	pairs in abelian with enough injective or projectives
	(which we prove to always induce derived equivalence), to abelian
	categories of global dimension at most two, to (co)hereditary torsion pairs
	over artin algebras (for which we give a purely combinatorial criterion for
	derived equivalence), and to study whether irreducible silting mutation acts
	transitively on two-term tilting complexes over a finite dimensional algebra.
\end{abstract}

\maketitle

\section*{Introduction}

The celebrated Tilting Theorem of Brenner and Butler \cite{bren-butl-80} provided examples
of abelian categories $\Acal$ and $\Bcal$ with torsion pairs
$\tau_\Acal=(\Tcal_\Acal,\Fcal_\Acal)$ in $\Acal$ and
$\tau_\Bcal=(\Tcal_\Bcal,\Fcal_\Bcal)$ in $\Bcal$, respectively, such that
$\Tcal_\Acal\simeq\Fcal_\Bcal$ and $\Fcal_\Acal\simeq\Tcal_\Bcal$. This led to
the introduction, by Happel, Reiten and Smalø \cite{happ-reit-smal-96}, of a general framework
for such a phenomenon. Given an abelian category $\Acal$ with an arbitrary
torsion pair $\tau=(\Tcal,\Fcal)$, the \emph{HRS-tilting} procedure yields an
abelian category $\Acal_\tau$ with a torsion pair $\bar\tau$ with this property
(see \S\ref{subsec:hrs-tilting} for the details). By construction, this abelian
category lies in the bounded derived category $\D^b(\Acal)$ of $\Acal$, and
therefore it is natural to ask how $\D^b(\Acal)$ compares with
$\D^b(\Acal_\tau)$. The torsion pair $\tau$ is said to induce \emph{derived
equivalence} if these two derived categories are triangle equivalent, in a
natural way (see \S\ref{subsec:derived-equivalence}).

There are several reasons for wanting to know whether a certain torsion pair
induces derived equivalence. For example, if the torsion pair comes from
(co)silting theory, this will detect whether the associated (co)silting object is
actually (co)tilting \cite[Prop.~5.1]{psar-vito-18}.
Another reason is the fact that several classes of algebras are closed under
derived equivalence: for example, piecewise-hereditary algebras (by definition),
self-injective algebras \cite[Thm.~A.5]{aiha-13}, \cite[Thm.~2.1]{alno-rick-13},
symmetric algebras \cite[Cor.~5.3]{rick-91}, Brauer graph algebras
\cite{anti-zvon-22}. When an HRS-tilt from $\mod\Lambda$ produces another category of modules over
an algebra (which happens when the torsion pair is \emph{functorially finite})
recognising that it gives derived equivalence guarantees that the new algebra is
still in the class. As a last example, HRS-tilting a category of modules $\Mod
A$ can produce an abelian category $\Hcal$ which is not a category of modules
anymore. However, if we have derived equivalence (which under some assumptions
can be promoted to the level of unbounded derived categories), we can infer
properties of the derived category $\D(\Hcal)\simeq\D(A)$, such as the fact that
it is compactly generated (this was a key ingredient in
\cite[\S3.3]{hrbe-pavo-23}).

In 2018, Chen, Han and Zhou \cite{chen-han-zhou-19} introduced a homological
criterion which completely characterises the torsion pairs in an abelian
category which induce derived equivalence. This criterion has the advantage of
being internal to the abelian category, instead of referring to the derived
category. The goal of this paper is to showcase the practical application of
this criterion in various contexts. We want to convice the reader that it proves
to be a useful computational tool in detecting derived equivalences in
several concrete situations, while at the same time it allows to draw
theoretical consequences which would be difficult to prove otherwise.

\medskip

\textbf{Structure of the paper.} In Section~\ref{sec:preliminaries} we recall the necessary
preliminaries and the criterion by Chen, Han and Zhou (Theorem~\ref{thm:chz}).

In Section~\ref{sec:general-applications}, we cover the first two applications, which take place in a general
setting. In \S\ref{subsec:stable}, we show that, in an abelian category with
enough injectives (respectively, projectives), all \emph{injective-splitting}
(respectively \emph{projective-splitting}) torsion pairs
induce derived equivalence. We give several examples of torsion pairs with
these properties (Example~\ref{ex:stable}), which
notably include all hereditary torsion pairs over a commutative noetherian ring
(Proposition~\ref{prop:stable-commutative-noetherian}), all torsion pairs over a
regular commutative noetherian ring (Proposition~\ref{prop:regular}), and all
torsion pairs of the category of quasi-coherent sheaves over the
projective line (Proposition~\ref{prop:sheaves}). In \S\ref{subsec:gldim2}, we
instead consider abelian categories of global dimension at most two. In this
setting, we show that if the category admits a noetherian generator
(respectively, artinian cogenerator) $M$, the torsion and torsion-free part of
$M$ are enough to determine whether a given torsion pair induces derived
equivalence (Proposition~\ref{prop:gldim2}). As a consequence, if such an object
$M$ exists, the class of torsion pairs giving derived equivalence is closed
under taking joins (respectively, meets) of arbitrary chains in the lattice of
torsion pairs (Proposition~\ref{prop:gldim2-chains}). For example, both these
results apply to the category of modules over a finite-dimensional algebra of
global dimension at most two.

In Section~\ref{sec:artin-applications}, we specialise to the category
$\mod(\Lambda)$ of modules over an artin (or finite-dimensional) algebra
$\Lambda$. In \S\ref{subsec:cohereditary}, we focus on (co)hereditary
torsion pairs, and obtain a purely combinatorial criterion to detect when they
induce derived equivalence. In particular, if $\Lambda$ is the path algebra of a
quiver with relations, our criterion has to do with prolonging certain paths in the
quiver, see Corollary~\ref{cor:path-algebra}, and it allows for estimates on the
number of (co)hereditary torsion pairs inducing derived equivalence.
As a last application of the criterion by Chen, Han and Zhou, in
\S\ref{subsec:2tilting} we consider the question of when the subset of the
lattice of torsion pairs consisting of those which induce derived equivalence is
a union of maximal chains. This is related to the question of whether irreducible
silting mutation acts transitively on the set of two-term tilting complexes. We
give a sufficient condition on $\Lambda$, which we call \emph{2-tilting
acyclicity}, which allows to construct ascending and
descending chains of torsion pairs inducing derived equivalence
(Proposition~\ref{prop:2-tilting-acyclic}). As a consequence, for this class of algebras
(which includes quasi-hereditary algebras and Nakayama algebras)
we obtain an analogue of the $\tau$-tilting-finiteness theorem by Demonet, Iyama
and Jasso \cite{demo-iyam-jass-19}, saying that when only finitely many
functorially finite torsion pairs induce derived equivalence, there are no other
torsion pairs inducing derived equivalence (Corollary~\ref{cor:dij-derived}).

\medskip

\textbf{Acknowledgments.} We would like to thank Rosanna Laking and Jorge
Vitória for discussions and suggestions regarding the topics of this paper. The
author was supported by NextGenerationEU under NRRP, Call PRIN 2022 No.\ 104 of
February 2, 2022 of Italian Ministry of University and Research; Project
2022S97PMY Structures for Quivers, Algebras and Representations (SQUARE).

\section{Preliminaries}\label{sec:preliminaries}

\subsection{Notation}

All subcategories we consider are strict and full. When we consider an abelian
category $\Acal$, we implicitly assume that it admits a bounded derived
category, denoted by $\D^b(\Acal)$. We write $\Acal(-,-)$ for the
$\Hom$-bifunctor of $\Acal$, and $\Acal^i(-,-)$ for the $i$-th
$\Ext$-bifunctor of $\Acal$. For subcategories $\Xcal,\Ycal\subseteq\Acal$, we
denote their span by
\[\Xcal\ast\Ycal:=\{\,a\in\Acal\mid \exists\;0\to x\to a\to y\to 0\text{ exact
with }x\in\Xcal, y\in\Ycal\,\}.\]
We use the same notation for the analogous notion in a triangulated category,
which uses distinguished triangles instead of short exact sequences.

For a subcategory $\Ccal\subseteq\Acal$, we denote its orthogonals by
\[\Ccal^\bot:=\{\,a\in\Acal\mid
\Acal(\Ccal,a)=0\,\}\qquad{}^\bot\Ccal:=\{\,a\in\Acal\mid\Acal(a,\Ccal)=0\,\}.\]

For a subcategory $\Ccal\subseteq\Acal$, the subcategory \emph{generated}
(respectively, \emph{cogenerated}) by $\Ccal$ is the subcategory
$\gen\Ccal\subseteq\Acal$ (respectively,
$\cogen\Ccal\subseteq\Acal$) consisting of quotients (respectively, subobjects)
of finite direct sums
of objects of $\Ccal$. If $\Ccal=\{c\}$, we write $\gen c$ and $\cogen c$. We
say that $\Ccal$ is (co)generating in $\Acal$ (or $c$ is a (co)generator of
$\Acal$) if $\Ccal$ (or $\{c\}$) (co)generates $\Acal$.

\subsection{Torsion pairs}

Let $\Acal$ be an abelian category.

\begin{dfn}
	A pair $(\Tcal,\Fcal)$ of subcategories of $\Acal$ is a \emph{torsion pair} if:
	\begin{enumerate}
		\item $\Acal(\Tcal,\Fcal)=0$;
		\item $\Acal=\Tcal\ast\Fcal$.
	\end{enumerate}
\end{dfn}

If $\tau=(\Tcal,\Fcal)$ is a torsion pair, for every object $a\in\Acal$ there is a
sequence $0\to t\to a\to f\to 0$ with $t\in\Tcal$ and $f\in\Fcal$, which is
unique up to isomorphism. We call this sequence the \emph{torsion sequence} of
$a$ with respect to $\tau$, and $t$ and $f$ the \emph{torsion} and
\emph{torsion-free part} of $a$, respectively. The classes $\Tcal$ and $\Fcal$
are the \emph{torsion} and \emph{torsion-free class} of $\tau$, respectively.

Each of the two classes of a torsion pair determines the other, as we have
$\Tcal^\bot=\Fcal$ and $\Tcal={}^\bot\Fcal$. From this it also follows that
$\Tcal$ is closed under extensions, quotients and existing coproducts, while
$\Fcal$ is closed under extensions, subobjects and existing products. Under some
assumptions on $\Acal$ (for example if it is well-powered and cocomplete, or
noetherian) torsion classes are characterised by these closure properties. The
same is true for torsion-free classes, under the dual hypotheses.

A torsion pair $\tau=(\Tcal,\Fcal)$ is said to be \emph{hereditary} if $\Tcal$
is closed under subobjects, and \emph{cohereditary} if $\Fcal$ is closed under
quotients. A \emph{TTF-triple} is a triple $(\Xcal,\Ycal,\Zcal)$ such that both
$(\Xcal,\Ycal)$ and $(\Ycal,\Zcal)$ are torsion pairs (necessarily cohereditary
and hereditary, respectively). In this case $\Ycal$ is called a
\emph{TTF-class}. Observe that both in a Grothendieck category and in a length
category, a subcategory $\Ycal$ is a TTF-class if and only if it is closed under
subobjects, quotients, extensions and existing products (and therefore also
existing coproducts).

\subsection{The lattice of torsion pairs}\label{subsec:lattice}
Let $\Acal$ be an abelian category. We denote by $\tors{\Acal}$ the class of all
torsion pairs of $\Acal$. If $\Acal$ is skeletally small, this is a proper set.
Otherwise it may not be (arguing as in \cite[\S8]{stan-10}), but this will not affect the way we are going
to use it.

We order torsion
pairs by inclusion of the torsion classes: for $\tau_1=(\Tcal_1,\Fcal_1)$ and
$\tau_2=(\Tcal_2,\Fcal_2)$ in $\tors\Acal$, we set $\tau_1\leq \tau_2$ if
$\Tcal_1\subseteq\Tcal_2$ (which is equivalent to $\Fcal_1\supseteq\Fcal_2$).
This makes $\tors{\Acal}$ into a poset, which has extremal elements
$\mathbf{0}:=(0,\Acal)$ and $\mathbf{1}:=(\Acal,0)$.

Whenever both torsion and torsion-free classes are characterised by closure
properties (for example, when $\Acal$ is a length category, or a Grothendieck
category), $\tors\Acal$ is also a complete lattice, with lattice
operations:
\[\textstyle\bigwedge_{i\in I}(\Tcal_i,\Fcal_i)=(\bigcap_{i\in I}\Tcal_i,(\bigcap_{i\in I}\Tcal_i)^\bot)
\qquad
\bigvee_{i\in I}(\Tcal_i,\Fcal_i)=({}^\bot(\bigcap_{i\in
I}\Fcal_i),\bigcap_{i\in I}\Fcal_i).\] 

We now introduce a new partition of the lattice $\tors\Acal$, induced by a fixed
object $a\in\Acal$.
For any short exact sequence $\epsilon\colon 0\to t\to a\to f\to 0$,
denote by $\tors(a/\epsilon)$ the class of torsion pairs
$(\Tcal,\Fcal)\in\tors\Acal$ for which $t\in\Tcal$ and $f\in\Fcal$. In other
words, these are the torsion pairs for which $\epsilon$ is the torsion sequence
of $a$, and therefore these sets form a partition of $\tors\Acal$ when we let
$\epsilon$ range over all short exact sequences with middle term $a$.
For example, this partition consists of exactly two (nonempty)
subsets if and only if $a$ is a nonzero \emph{torsion-simple} object, that is,
it is always torsion or torsion-free, for any torsion pair \cite{pavo-25a}.

Notice that the class $\tors(a/\epsilon)$ is clearly empty if $\Acal(t,f)\neq
0$. Now assume that the torsion pairs generated and cogenerated by an object
exist (this is the case if either torsion or torsion-free classes are
characterised by closure properties in $\Acal$). Then, if $\Acal(t,f)=0$, any
torsion pair $(\Tcal,\Fcal)\in\tors(a/\epsilon)$ will have $t\in\Tcal$ and
$f\in\Fcal$, and therefore by definition it will be above the torsion pair
generated by $t$ and below that cogenerated by $f$. This shows that
$\tors(a/\epsilon)$ is an interval, with those two pairs as the extremes.
For examples, see Figures~\ref{fig:partition} and~\ref{fig:kronecker}.

%{{{ A3 picture
\newcommand{\symb}[1] {%
	\IfEqCase{#1}{%
        {t}{\bullet}%
        {f}{\circ}%
				{m}{\ast}%
		}%
}
\newcommand{\ar}[6] %
	{\begin{smallmatrix}\symb#6\\[-2pt]\symb#4 \symb#5\\[-2pt]\symb#1 \symb#2 \symb#3\end{smallmatrix}}

\usetikzlibrary{decorations.text, fit, positioning, quotes, backgrounds,calc}
\begin{figure}
\begin{tikzpicture}[on top/.style={preaction={draw=white,-,line width=#1}},
on top/.default=6pt, yscale=.7]
	
	\node (At) at (0.5,4.5) {$\ar tttttt$};
	\node (Bt) at (-2.5,1.5) {$\ar tftmtt$};
	\node (Ct) at (0.5,2.5) {$\ar fttttt$};
	\node (Dt) at (2.5,2.5) {$\ar ttftmm$};
	\node (Et) at (-0.5,1.5) {$\ar fttftt$};
	\node (Ft) at (2.5,0.5) {$\ar ftftmm$};
	\node (Gt) at (-1.5,0.5) {$\ar fftftt$};

	\node (Ab) at (-0.5,-4.5) {$\ar ffffff$};
	\node (Bb) at (2.5,-1.5) {$\ar ftffmf$};
	\node (Cb) at (-0.5,-2.5) {$\ar fftfff$};
	\node (Db) at (-2.5,-2.5) {$\ar tffmfm$};
	\node (Eb) at (0.5,-1.5) {$\ar fftftf$};
	\node (Fb) at (-2.5,-0.5) {$\ar tftmfm$};
	\node (Gb) at (1.5,-0.5) {$\ar fttftf$};
	
	\draw[-] (At) -- (Bt);
	\draw[-] (At) -- (Ct);
	\draw[-] (At) -- (Dt);
	\draw[-] (Bt) -- (Gt);
	\draw[-] (Bt) -- (Fb);
	\draw[-] (Ct) -- (Et);
	\draw[-] (Ct) -- (Ft);
	\draw[-] (Dt) -- (Ft);
	\draw[-] (Et) -- (Gt);

	\draw[-] (Ab) -- (Bb);
	\draw[-] (Ab) -- (Cb);
	\draw[-] (Ab) -- (Db);
	\draw[-] (Bb) -- (Gb);
	\draw[-] (Bb) -- (Ft);
	\draw[-] (Cb) -- (Eb);
	\draw[-] (Cb) -- (Fb);
	\draw[-] (Db) -- (Fb);
	\draw[-] (Eb) -- (Gb);

	\draw[-] (Gt) -- (Eb);
	\draw[-] (Et) -- (Gb);
	\draw[-] (Dt) -- (Db);

	\begin{scope}[on background layer]
		\draw[fill=lightgray,color=lightgray,line width=24pt,line cap=round,rounded corners=1pt]
			($(At.center)+(90:4pt)$) -- ($(Gb.center)+(315:4pt)$) --
			($(Eb.center)+(270:4pt)$) -- ($(Bt.center)+(135:4pt)$) -- cycle;
			;
%		\draw[color=lightgray,line width=24pt,line cap=round,rounded corners=1pt]
%			($(Bt.center)+(135:4pt)$) -- ($(Gt)+(315:4pt)$);
		\draw[color=lightgray,line width=24pt,line cap=round,rounded corners=1pt]
			($(Dt.center)+(90:4pt)$) -- ($(Bb.center)+(270:4pt)$);
		\draw[color=lightgray,line width=24pt,line cap=round,rounded corners=1pt]
			($(Fb.center)+(135:4pt)$) -- ($(Cb.center)+(315:4pt)$);
		\draw[color=lightgray,line width=24pt,line cap=round,rounded corners=1pt]
			($(Ab.center)+(315:1pt)$) -- ($(Db.center)+(315:1pt)$);
	\end{scope}
\end{tikzpicture}
\caption{The Hasse quiver of $\tors(\Lambda)$, for the algebra
$\Lambda=k(\bullet\to\bullet\to\bullet)$. Each torsion pair is represented with the
shape of the Auslander--Reiten quiver of $\Lambda$, marking the torsion modules with
$\bullet$, the torsion-free modules with $\circ$ and the others with $\ast$. If
we choose $M$ to be the direct sum of the indecomposable injectives which are
not projective, the corresponding partition of $\tors(\Lambda)$ into the sets
$\tors(M/\epsilon)$ is shaded in gray. By definition, each torsion pair in a
given class of this partition induces the same torsion sequence for $M$.}\label{fig:partition} 
\end{figure}
%}}}

\subsection{$t$-structures}

In $\D^b(\Acal)$, as a triangulated version of the notion of torsion pair, we
will consider the following notion.

\begin{dfn}
	A pair $(\Ucal,\Vcal)$ of subcategories of $\D^b(\Acal)$ is a
	\emph{$t$-structure} if:
	\begin{enumerate}
		\item $\Ucal[1]\subseteq\Ucal$ (or equivalently, $\Vcal[-1]\subseteq\Vcal$);
		\item $\D^b(\Acal)(\Ucal,\Vcal)=0$;
		\item $\D^b(\Acal)=\Ucal\ast\Vcal$.
	\end{enumerate}
\end{dfn}

As for torsion pairs, for any $z\in\D^b(\Acal)$ there is a triangle $u\to z\to
v\to u[1]$ with $u\in\Ucal$ and $v\in\Vcal$, which is unique up to unique isomorphism.
The classes $\Ucal$ and $\Vcal$ are the \emph{aisle} and the \emph{coaisle} of
the $t$-structure, and $u$ and $v$ are respectively the \emph{left} and the
\emph{right truncation} of $z$.

The \emph{heart} of a $t$-structure $(\Ucal,\Vcal)$ is the subcategory
$\Ucal[-1]\cap\Vcal$, which is abelian \cite[Thm.~1.3.6]{bbd-81}. The short
exact sequences of this category are precisely the triangles of $\D^b(\Acal)$
whose terms lie in the heart.

As an example, the \emph{standard $t$-structure} of $\D^b(\Acal)$ is the pair
$(\Ucal_0,\Vcal_0)$ given by:
\begin{align*}
	\Ucal_0&:=\{\,x\in\D^b(\Acal) \mid H^i(x)=0 \text{ for every }i\geq 0\,\} \\
	\Vcal_0&:=\{\,x\in\D^b(\Acal)\mid H^i(x)=0 \text{ for every }i<0\,\}
\end{align*}
with heart $\Ucal_0[-1]\cap\Vcal_0\simeq \Acal$.

\subsection{HRS-tilting}\label{subsec:hrs-tilting}

Let $\Acal$ be an abelian category, $(\Ucal,\Vcal)$ a $t$-structure in
$\D^b(\Acal)$ with heart $\Hcal=\Ucal[-1]\cap\Vcal$ and $\tau=(\Tcal,\Fcal)$ a
torsion pair in $\Hcal$. In \cite{happ-reit-smal-96}, Happel, Reiten and Smalø introduced the
following construction.

\begin{dfn}
	There is a $t$-structure in $\D^b(\Acal)$:
	\[(\Ucal_\tau,\Vcal_\tau)=(\Ucal\ast\Tcal,\Fcal\ast\Vcal[-1])\text{ with heart }
	\Hcal_\tau=\Fcal\ast\Tcal[-1],\]
	which is called the (right) \emph{HRS-tilt of
	$\Hcal$ at $\tau$ in $\D^b(\Acal)$}.
\end{dfn}

This construction yields a bijection between the class of torsion pairs of
$\Hcal$ and that of $t$-structures $(\Ucal',\Vcal')$ of $\D^b(\Acal)$ with the
property that $\Ucal\subseteq \Ucal'\subseteq\Ucal[-1]$. The
inverse of this bijection associates to such a $t$-structure $(\Ucal',\Vcal')$ the
pair $(\Ucal'\cap\Hcal,\Vcal'\cap\Hcal)$, which is a torsion pair in $\Hcal$.

For a torsion pair $\tau=(\Tcal,\Fcal)$ in $\Hcal$, observe that since we have
$\Ucal\subseteq\Ucal_\tau\subseteq\Ucal[-1]$, we also have
$\Ucal_\tau\subseteq\Ucal[-1]\subseteq\Ucal_\tau[-1]$. Therefore, the shifted
$t$-structure $(\Ucal[-1],\Vcal[-1])$ is obtained by HRS-tilting
$(\Ucal_\tau,\Vcal_\tau$), at the following torsion pair in
$\Hcal_\tau=\Fcal\ast\Tcal[-1]$:
\[(\Ucal[-1]\cap\Hcal_\tau,\Vcal[-1]\cap\Hcal_\tau)=(\Fcal,\Tcal[-1])=:\bar\tau.\]
This torsion pair was considered already in \cite{happ-reit-smal-96}.

Now, let $\sigma$ be another torsion pair in $\Hcal$. Observe that we have
$\tau\leq\sigma$ if and only if
$\Ucal_\tau\subseteq\Ucal_\sigma\subseteq\Ucal[-1]$. By adding the inclusion
$\Ucal\subseteq\Ucal_\tau$ and its shift, we see that this is equivalent to
having the following chain of inclusions:
\[\Ucal\subseteq\Ucal_\tau\subseteq\Ucal_\sigma\subseteq\Ucal[-1]\subseteq\Ucal_\tau[-1],\]
that is, to the fact that $(\Ucal_\sigma,\Vcal_\sigma)$ can be obtained by
HRS-tilting from both $(\Ucal,\Vcal)$ and $(\Ucal_\tau,\Vcal_\tau)$. From the
point of view of $\Hcal_\tau$, this is the HRS-tilt at a pair $\bar\sigma\leq
\bar\tau$. This argument shows that there is an order preserving bijection
between the intervals:
\[\{\sigma\text{ in }\Hcal\mid \tau\leq \sigma\leq \mathbf{1}\}
\;\longleftrightarrow\;
\{\bar\sigma\text{ in }\Hcal_\tau\mid \mathbf{0}\leq\bar\sigma\leq \bar\tau \},
\]
see \cite[Prop.~7.5]{ange-laki-stov-vito-22}.
As mentioned, we also have $\Hcal[-1]=(\Hcal_\tau)_{\bar\tau}$. Applying this
isomorphism of intervals (with a shift), also yields another isomorphism of
intervals:
\[
\{\bar\sigma\text{ in }\Hcal_\tau\mid \bar\tau\leq
\bar\sigma\leq\mathbf{1}\}\;\longleftrightarrow\;
\{\sigma\text{ in }\Hcal[-1]\mid \mathbf{0}\leq\sigma[-1]\leq \tau[-1]\}=
\{\sigma\text{ in }\Hcal\mid \mathbf{0}\leq\sigma\leq \tau\}.
\]

\subsection{Derived equivalence and the CHZ
criterion}\label{subsec:derived-equivalence}

Let $\Acal$ be an abelian category, and $\tau=(\Tcal,\Fcal)$ a torsion pair in
$\Acal$. The $t$-structure $(\Ucal_\tau,\Vcal_\tau)$ with heart $\Acal_\tau$ obtained by HRS-tilting the
standard $t$-structure of $\D^b(\Acal)$ brings along a \emph{realisation
functor}, which is a triangle functor $\D^b(\Acal_\tau)\to \D^b(\Acal)$
extending the inclusion $\Acal_\tau\hookrightarrow\D^b(\Acal)$ \cite{bbd-81}. We say
that $\tau$ \emph{induces derived equivalence} if the realisation functor is an
equivalence $\D^b(\Acal_\tau)\simeq\D^b(\Acal)$.
We will write $\tors^d{\Acal}\subseteq\tors{\Acal}$ for the class of torsion
pairs inducing derived equivalence.

\begin{thm}{{\cite[Thm.~A]{chen-han-zhou-19}}}\label{thm:chz}
	Let $\Acal$ be an abelian category, and $\tau=(\Tcal,\Fcal)$ a torsion pair in
	$\Acal$. Then $\tau$ induces derived equivalence if and only if for every
	object $a\in\Acal$ there is an exact sequence
	\[0\to f_0\to f_1\to a\to t_0\to t_1\to 0,\]
	with $t_i\in\Tcal$ and $f_i\in\Fcal$, representing the zero element of
	$\Acal^3(t_1,f_0)$.
\end{thm}

\section{General applications of the CHZ criterion}
\label{sec:general-applications}

\subsection{First application: injective-splitting and projective-splitting
torsion pairs}
\label{subsec:stable}

From the criterion of Theorem~\ref{thm:chz}, we deduce the following
consequence, which is essentially a rephrasing of
\cite[Cor.~4.1]{chen-han-zhou-19}. Recall that we say that a class
$\Ccal\subseteq\Acal$ is generating in $\Acal$ if $\Acal=\gen\Ccal$, and cogenerating
in $\Acal$ if $\Acal=\cogen\Ccal$.

\begin{cor}\label{cor:F*T}
	If $\tau=(\Tcal,\Fcal)$ is such that $\Fcal\ast\Tcal$ is generating or
	cogenerating in $\Acal$, then $\tau$ induces derived equivalence.
\end{cor}

\begin{proof}
	For completeness, we prove the cogenerating case, the other being dual. For any $a\in\Acal$, by
	assumption we have a monomorphism $a\hookrightarrow e$ for
	$e\in\Fcal\ast\Tcal$ (as this class is closed under finite direct sums). This yields a pullback diagram:
	\[\begin{tikzcd}
		0 \arrow{r} & f_1 \arrow[hook]{d} \arrow{r}
			\arrow[phantom,description,very near start, "\lrcorner"]{dr} &
			a \arrow[hook]{d} \arrow{r} & t_0 \arrow[equal]{d} \\
		0 \arrow{r} & f \arrow{r} & e \arrow{r} & t \arrow{r} & 0,
	\end{tikzcd}\]
	where the bottom row, with $f\in\Fcal$ and $t\in\Tcal$, witnesses the
	assumption that $e$ lies in $\Fcal\ast\Tcal$. Since $\Fcal$ is closed under
	subobjects, we also
	have $f_1\in\Fcal$. Completing the top row with the cokernel $t_1$, which
	is a quotient of $t_0:=t\in\Tcal$, yields a
	sequence as in Theorem~\ref{thm:chz}, with $f_0=0$.
\end{proof}

In this subsection we give a new application of Corollary~\ref{cor:F*T}.
Let $\Acal$ be an abelian category, and $\Ccal\subseteq\Acal$ a class of
objects. Say that a torsion pair $(\Tcal,\Fcal)$ \emph{splits
$\Ccal$} (or is \emph{$\Ccal$-splitting}) if the torsion sequence of every
object $c\in\Ccal$ splits.
This terminology is modelled on the definition of \emph{brick-splitting} torsion pairs
recently introduced in \cite{asai-iyam-mous-paqu-25} (which is what one obtains
by taking $\Ccal$ to be the class of bricks). We have the following.

\begin{prop}\label{prop:cogen-splitting}
	Let $\Acal$ be an abelian category. If a torsion pair splits a
	generating or cogenerating class $\Ccal\subseteq\Acal$, then it induces
	derived equivalence.
\end{prop}

\begin{proof}
	If $(\Tcal,\Fcal)$ splits a generating or cogenerating class
	$\Ccal\subseteq\Acal$, since we have
	$\Ccal\subseteq\Fcal\oplus\Tcal\subseteq\Fcal\ast\Tcal$, we conclude by
	Corollary~\ref{cor:F*T}.
\end{proof}

The first example of this situation we consider is when we take
$\Ccal:=\Inj\Acal$ to be the class of injectives and assume that it is
cogenerating, that is, that $\Acal$ has enough injectives. We observe
immediately that if in addition $\Acal$ has injective envelopes, this recovers
the classical notion of stable torsion pair (see \cite[\S{}VI.7]{sten-75}), as
follows.

\begin{lemma}
	Let $\Acal$ be a category with injective envelopes. Then a torsion pair
	$(\Tcal,\Fcal)$ is injective-splitting if and only if it is \emph{stable},
	that is, $\Tcal$ is closed under injective envelopes.
\end{lemma}

\begin{proof}
	$(\Rightarrow)$ For $t_0\in\Tcal$, consider the injective envelope $e(t_0)$ and its
	torsion sequence $0\to t\to e(t_0)\to f\to 0$, which splits by assumption.
	Since $t_0$ belongs to $\Tcal$, it is contained in the torsion part $t$ of
	$e(t_0)$, which is therefore an essential direct summand. Thus we have have
	$e(t_0)=t\in\Tcal$.

	$(\Leftarrow)$ Assume that $\Tcal$ is closed under injective envelopes, and
	consider the torsion sequence $0\to t\to e\to f\to 0$ of an injective
	$e\in\Gcal$. Then $e$ contains as a direct summand the injective envelope
	$e(t)$ of $t$. Since $e(t)$ belongs to $\Tcal$ by assumption, however, it must
	be contained in the torsion part $t$ of $e$: therefore we have that $t=e(t)$
	is injective, and the torsion sequence splits.
\end{proof}

\begin{ex}\label{ex:stable}
	We give a few examples of injective-splitting (in fact, stable) torsion pairs,
	both hereditary and not. By Proposition~\ref{prop:cogen-splitting}, they all
	induced derived equivalence.

	(1) Let $\Gcal$ be a locally noetherian Grothendieck category of global
	dimension one. Then the class $\Inj\Gcal$ of injectives is closed under
	extensions (which are split), coproducts (by local noetherianity) and
	quotients
	(by global dimension), and therefore it is the torsion class of a torsion pair
	$(\Inj\Gcal,\Inj\Gcal^\bot)$. For example, for $\Gcal=\Mod(\Zbb)$ this
	construction yields the divisible-reduced torsion pair. This pair is obviously
	stable, but not hereditary if $\Gcal$ is not semisimple.

	More generally, if $\Gcal$ is a Grothendieck category, every torsion pair
	whose torsion class contains $\Inj\Gcal$ (which we could dub
	\emph{cofaithful}) is obviously stable.

	(2) Consider the torsion pair of $\mod(k\Dbb_4)$ generated by the
	indecomposable representations:
	\[\begin{tikzcd}[column sep=small, row sep=0] && k \\ 0 \ar{r} & k
	\ar{ur}\ar{dr} \\ && 0\end{tikzcd}\qquad\text{and}\qquad
	\begin{tikzcd}[column sep=small, row sep=0] && 0\phantom{.} \\ k \ar{r} & 0
	\ar{ur}\ar{dr} \\ && 0.\end{tikzcd}
	\]
	Drawing the Auslander--Reiten quiver of $k\Dbb_4$ with the dimension vectors
	of each indecomposable representation, and marking torsion modules with a
	shade and torsion-free modules with a frame, this pair looks as
	follows.
	\newcommand{\tframe}[1]{\colorbox{lightgray}{$#1$}}
	\newcommand{\fframe}[1]{\fbox{$#1$}}
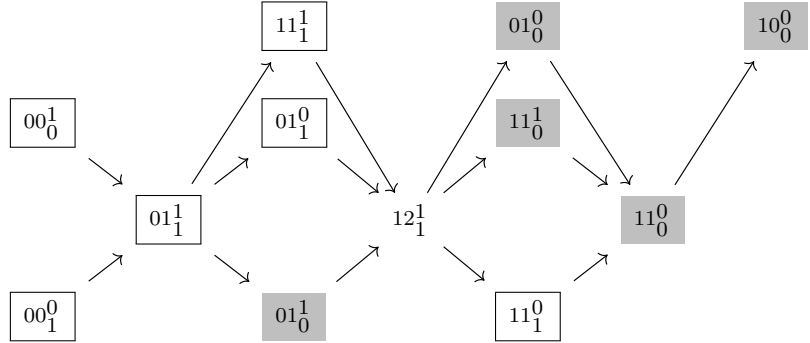
	\[\begin{tikzcd}[sep=small]
		& & \fframe{\dfour1111} \ar{ddr} && \tframe{\dfour0100} \ar{ddr} &&
			\tframe{\dfour1000} \\
		\fframe{\dfour0010} \ar{dr} && \fframe{\dfour0101} \ar{dr} &&
			\tframe{\dfour1110} \ar{dr} \\
		& \fframe{\dfour0111} \ar{ur}\ar{dr}\ar{uur} && \dfour1211 \ar{ur}\ar{uur} \ar{dr} &&
			\tframe{\dfour1100} \ar{uur} \\
		\fframe{\dfour0001} \ar{ur} && \tframe{\dfour0110} \ar{ur} &&
			\fframe{\dfour1101} \ar{ur}
	\end{tikzcd}\]
	From this, it is easy to see that this torsion pair is not hereditary, it is
	not split (that is, not all torsion sequences split), but it is stable.

	(3) Over a commutative noetherian ring $R$, every hereditary torsion pair is
	stable. More generally, this is true over any right noetherian ring $A$ that
	is \emph{fully bounded} and \emph{right Artin--Rees}, whose definitions we know
	recall.
	We say that a ring $A$ is \emph{fully bounded} if every indecomposable injective right
	$A$-module is the injective envelope of some $A/\pf$, for $\pf$ a prime ideal
	\cite[\S VII.2]{sten-75}. For example, any finitely generated algebra over a
	commutative noetherian ring is fully bounded
	\cite[Example~\S{}VII.2.4]{sten-75}. On the other hand, following \cite[\S
	VII.4]{sten-75}, we say that $A$ is \emph{right
	Artin--Rees} (as a right module over itself) if for every right ideal $I\leq
	A$, every two-sided ideal $\mathfrak{a}\leq A$ and every $n\geq 1$ there
	exists $h(n)\geq 1$ such that $\mathfrak{a}^{h(n)}\cap I\subseteq
	I\mathfrak{a}^n$. Commutative noetherian rings are right Artin--Rees by
	\cite[Prop.~VII.4.5]{sten-75}.

	Now, let $A$ be a right noetherian, fully bounded, right Artin--Rees ring. A
	hereditary torsion pair in $\Mod(A)$ is described by a \emph{Gabriel topology}
	\cite{sten-75}. Since $A$ is fully bounded, every Gabriel topology is
	\emph{bounded} by \cite[Thm.~VII.3.4]{sten-75}, and therefore also
	\emph{stable} since $A$ is right Artin--Rees
	\cite[Thm.~VII.4.4(c$\Leftrightarrow$e)]{sten-75}. This means by definition
	\cite[after Prop.~VI.7.1]{sten-75} that every hereditary torsion pair in
	$\Mod(A)$ is stable in our sense.

	Besides commutative noetherian rings, another example of right noetherian,
	fully bounded, right Artin--Rees rings are \emph{Azumaya algebras} (also
	called \emph{central separable algebras}), by
	\cite[Ex.~VII.4.2]{sten-75}.
	These are the rings $A$ which are finitely generated algebras over their noetherian
	center $R$ and which are projective over $A\otimes_R A^{op}$
	\cite{ford-17,ausl-gold-60}.

	(4) Over any ring $A$, the hereditary torsion pair generated by the cyclic
	modules $A/I$, where $I$ is an essential right ideal of $A$, is called the
	\emph{Goldie} torsion pair (see \cite[\S{}VI.6.2]{sten-75}), and it is stable
	\cite[Prop.~VI.7.3]{sten-75}.

	(5) By \cite[Prop.~3.27]{pavo-25a}, in the category $\mathsf{Qcoh}(\Pbb^1)$ of
	quasi-coherent sheaves over the projective line, every indecomposable
	injective is torsion-simple. Since this category
	is locally noetherian, so that every injective is a coproduct of
	indecomposable injectives, this shows that every torsion pair in
	$\mathsf{Qcoh}(\Pbb^1)$ is injective-splitting.
\end{ex}

As mentioned, Proposition~\ref{prop:cogen-splitting} applies to all the entries
of Example~\ref{ex:stable}. In particular, item (3) yields a neat new proof of
\cite[Cor.~5.11]{pavo-vito-21}.

\begin{prop}\label{prop:stable-commutative-noetherian}
	Let $A$ be a right noetherian, fully bounded, right Artin--Rees ring (such as
	a commutative noetherian ring, or more generally an Azumaya algebra). Then
	every hereditary torsion pair of $\Mod(A)$ induces derived equivalence.
\end{prop}

\begin{proof}
	As explained in Example~\ref{ex:stable}(3), every hereditary torsion pair in
	$\Mod(A)$ is stable, and the result follows from
	Proposition~\ref{prop:cogen-splitting}.
\end{proof}

We also point out the application to item (5) of Example~\ref{ex:stable}.

\begin{prop}\label{prop:sheaves}
	Every torsion pair in $\mathsf{Qcoh}(\Pbb^1)$ induces derived equivalence.
\end{prop}

\begin{proof}
	As argued in the Example, every torsion pair in $\mathsf{Qcoh}(\Pbb^1)$ is
	injective-splitting.
\end{proof}

We now turn to projective-splitting torsion pairs.
Recall that a commutative ring $R$ is \emph{regular} if we have $\gldim(\Mod
R_\pf)<\infty$ for every (maximal) prime $\pf$ of $R$. This includes all
commutative rings $R$ with $\gldim(\Mod R)<\infty$.
For noetherian regular rings, we can considerably strengthen
Proposition~\ref{prop:stable-commutative-noetherian}, by applying
Proposition~\ref{prop:cogen-splitting} with $\Ccal:=\Proj R$.

\begin{prop}\label{prop:regular}
	Let $R$ be a regular commutative noetherian ring. Then every torsion pair is
	projective-splitting, and in particular it induces derived equivalence.
\end{prop}

\begin{proof}
%	If $R$ is regular, then its localisations at prime ideals are also regular
%	\cite{}, and therefore they are reduced domains by \cite{}. This shows that
%	below any prime ideal of $R$ there is a unique minimal prime. Equivalently,
%	any two minimal primes are comaximal, that is, their sum is the whole $R$. On
%	the other hand, the intersection of the minimal ideals of $R$, which is the
%	nilradical, must be zero because a regular ring is reduced \cite{}. The
%	Chinese Remainder Theorem then shows that $R$ is isomorphic to the finite
%	product of integral domains $R/\pf_1\times\cdots\times R/\pf_n$, for
%	$\pf_1,\dots,\pf_n$ the minimal primes.
	By \cite[Cor.~2.2.20]{brun-herz-93}, every regular commutative ring is a
	product of domains, and then necessarily finitely many if it is noetherian,
	$R\simeq D_1\times \cdots \times D_n$ (in fact, one can take $D_i:=R/\pf_i$
	where $\pf_i$ are the minimal primes of $R$).
	We then have a decomposition $\Mod(R)\simeq \Mod(D_1)\times \cdots\times
	\Mod(D_n)$. As a consequence, on one hand we have that $\Proj(R)
	=\Proj(D_1)\times\cdots\times \Proj(D_n)$, while on the other any
	torsion pair of $\Mod(R)$ has the form
	$(\Tcal_1\times\cdots\times\Tcal_n,\Fcal_1\times\cdots\times\Fcal_n)$, for
	torsion pairs $(\Tcal_i,\Fcal_i)$ of $\Mod(D_i)$. To prove that every
	torsion pair in $\Mod(R)$ is projective-splitting, it is then sufficient to
	prove the same claim in $\Mod(D_i)$.

	Since $D_i$ is a domain, the zero ideal is prime in $D_i$, and therefore the regular module $D_i\simeq D_i/0$ is
	torsion-simple in $\Mod(D_i)$, by \cite[Corollary~3.11]{pavo-25a}.
	It follows that all the projective $D_i$-modules are also torsion-simple by
	\cite[Prop.~2.5]{pavo-25a}, since $\Proj D_i=\Add D_i$. \emph{A fortiori},
	this means that every torsion pair splits projectives, concluding the proof.
\end{proof}

\subsection{Second application: global dimension $2$}
\label{subsec:gldim2}

Now we consider the case of an abelian category $\Acal$ with
global dimension $\gldim\Acal\leq 2$. Under this assumption, the criterion of
Theorem~\ref{thm:chz} simplifies considerably, as the requirement that the
sequence vanishes in $\Acal^3(t_1,f_0)$ is automatically satisfied.
We then have the following.

\begin{cor}
	Assume that $\gldim\Acal\leq 2$. Then a torsion pair $(\Tcal,\Fcal)$ in
	$\Acal$ induces derived equivalence if and only if $\Acal=\gen\Fcal\ast\cogen\Tcal$.
\end{cor}

\begin{proof}
	An object $a\in\Acal$ lies in $\gen\Fcal\ast\cogen\Tcal$ precisely if the exists
	an exact sequence $f_1\to a\to t_0$ with $f_1\in\Fcal$ and $t_0\in\Tcal$,
	which then can be completed to a sequence as in Theorem~\ref{thm:chz}.
\end{proof}

Observe that this span has the following closure properties.

\begin{lemma}
	Let $(\Tcal,\Fcal)$ be a torsion pair in an abelian category. Then
	the subcategories $\gen\Fcal$, $\cogen\Tcal$ and $\gen\Fcal\ast\cogen\Tcal$
	are closed under subobjects, quotients and finite direct sums.
\end{lemma}

\begin{proof}
	$\gen\Fcal$ is closed under quotients and direct sums by construction, and it
	is closed under subobject by an easy pullback argument, using that $\Fcal$ is
	closed under subobjects. The argument for $\cogen\Tcal$ is dual. If
	$\Xcal$ and $\Ycal$ are subcategories closed under subobjects,
	their span $\Xcal\ast\Ycal$ is also closed under subobjects by another
	pullback argument. So $\gen\Fcal\ast\cogen\Tcal$ is closed under subobjects,
	and similary under quotients. Closure under direct sums is clear.
\end{proof}

In particular, if $M$ is a (co)generator of $\Acal$, meaning that $\Acal=\gen M$
or $\Acal=\cogen M$, then $\gen\Fcal\ast\cogen\Tcal$ equals $\Acal$ if and only
if it contains $M$.

For an object $a\in\Acal$ and a subcategory $\Ccal\subseteq\Acal$, define the
\emph{trace} and the \emph{reject} of $\Ccal$ in $a$ as:
\begin{align*}
	\tr_\Ccal a&:= \text{the largest subobject } x\leq a \text{ with } x\in\gen\Ccal,\\
	\rej_\Ccal a&:= \text{the smallest subobject } x\leq a \text{ with } a/x\in\cogen\Ccal.
\end{align*}
The trace (respectively, the reject) is well-defined as soon as $a$ is
noetherian (respectively, artinian).
For an object $c$, we write
$\tr_c:=\tr_{\gen c}$ and $\rej_c:=\rej_{\cogen c}$.

\begin{lemma}
	Let $\Acal$ be an abelian category and $(\Tcal,\Fcal)$ a torsion pair. For a
	noetherian (respectively, artinian) object $a\in\Acal$, we have that
	$a\in\gen\Fcal\ast\cogen\Tcal$ if and only if $a/\tr_\Fcal a\in\cogen\Tcal$
	(respectively, $\rej_\Tcal a\in\gen\Fcal$).
\end{lemma}

\begin{proof}
	We prove the claim for $a$ noetherian, as the other is dual. The implication
	$(\Leftarrow)$ is clear, given the short exact sequence $0\to \tr_\Fcal a\to
	a\to a/\tr_\Fcal a\to 0$. For the converse ($\Rightarrow$), let $0\to x\to a\to y\to 0$ be a
	short exact sequence with $x\in\gen\Fcal$ and $y\in\cogen\Tcal$. Then by
	construction we have $x\subseteq\tr_\Fcal a$, and therefore we have an
	epimorphism $y\twoheadrightarrow a/\tr_\Fcal a$. The latter then belongs to
	$\cogen\Tcal$, as this class is closed under quotients.
\end{proof}

Let $M$ be a generator (respectively, a cogenerator) of $\Acal$, and
let $0\to T\to M\to F\to 0$ be its torsion sequence with respect to
$(\Tcal,\Fcal)$. Then we have that $\gen\Fcal=\gen F$ (respectively,
$\cogen\Tcal=\cogen T$), as any epimorphism $M^n\to f$ for $f\in\Fcal$ factors
through the torsion-free part $F^n$ of $M^n$. Combining all the observations
made so far, we obtain the following. We remark that if $\Acal$ has a noetherian
generator (respectively, an artinian cogenerator), it is a noetherian
(respectively, artinian) category.

\begin{prop}\label{prop:gldim2}
	Let $\Acal$ be an abelian category with $\gldim\Acal\leq 2$, $(\Tcal,\Fcal)$ a
	torsion pair in $\Acal$ and $0\to T\to M\to F\to 0$ the torsion sequence of an
	object $M$. Then:
	\begin{enumerate}
		\item if $M$ is a noetherian generator, then $(\Tcal,\Fcal)\in\tors^d\Acal$
			if and only if $M/\tr_FM\in\cogen\Tcal$;
		\item if $M$ is an artinian cogenerator, then $(\Tcal,\Fcal)\in\tors^d\Acal$
			if and only if $\rej_TM\in\gen\Fcal$.
	\end{enumerate}
\end{prop}

\begin{proof}
	We summarise the argument for (1), as (2) is dual. Since $\gldim\Acal\leq 2$, the
	pair $(\Tcal,\Fcal)$ gives derived equivalence if and only if
	$\gen\Fcal\ast\cogen\Tcal$ equals $\Acal$, and this is the case if and only if
	it contains the generator $M$. This happens if and only if
	$M/\tr_\Fcal M=M/\tr_FM$ lies in $\cogen\Tcal$.
\end{proof}

This has the following consequence, involving the sets $\tors(M/\epsilon)$
introduced in \S\ref{subsec:lattice}.

\begin{cor}\label{cor:upper-lower}
	Let $\Acal$ be an abelian category with $\gldim\Acal\leq2$, and
	$\epsilon\colon 0\to T\to M\to F\to 0$ a exact sequence. Then:
	\begin{enumerate}
		\item if $M$ is a noetherian generator, then
			$\tors^d\Acal\cap\tors(M/\epsilon)$ is an upper-set in
			$\tors(M/\epsilon)$;
		\item if $M$ is an artinian cogenerator, then
			$\tors^d\Acal\cap\tors(M/\epsilon)$ is a lower-set in
			$\tors(M/\epsilon)$.
	\end{enumerate}
\end{cor}

\begin{proof}
	Fix $\varepsilon$ as in the statement. Then, for any torsion pair
	$(\Tcal,\Fcal)\in\tors(M/\epsilon)$, the objects $M/\tr_\Fcal M=M/\tr_F M$ and
	$\rej_\Tcal M=\rej_T M$
	are fixed, as they only depend on $T$ and $F$. From
	Proposition~\ref{prop:gldim2}, we deduce
	that a torsion pair $(\Tcal,\Fcal)\in\tors(M/\epsilon)$ gives derived
	equivalence if and only if $\Tcal$ (respectively, $\Fcal$) is large enough to
	cogenerate $M/\tr_FM$ (respectively, to generate $\rej_TM$). The claim follows
	immediately.
\end{proof}

\begin{cor}\label{cor:derived-partition}
	Let $\Acal$ be a length abelian category with $\gldim\Acal\leq 2$, and let
	$M\in\Acal$ be both a generator and a cogenerator (for example, $\Acal=\mod A$
	for an artin algebra $A$ with $\gldim A\leq 2$ and $M=A\oplus DA$).
	Then, for every short exact
	sequence $\epsilon\colon 0\to T\to M\to F\to 0$, the set $\tors(M/\epsilon)$
	is contained either in $\tors^d\Acal\subseteq\tors\Acal$ or in its complement.
\end{cor}

\begin{proof}
	Since $\Acal$ is a length category, the sets $\tors(M/\epsilon)$ are
	intervals, so their only subsets which are at the same time upper- and lower-
	are the empty set and themselves. By Corollary~\ref{cor:upper-lower}, the
	intersection $\tors(M/\epsilon)\cap\tors^d\Acal$ must then have this form.
\end{proof}

Another interesting consequence is the following result.

\begin{prop}\label{prop:gldim2-chains}
	Let $\Acal$ be an abelian category with $\gldim\Acal\leq2$. Then:
	\begin{enumerate}
		\item if $\Acal$ has a noetherian generator, $\tors^d\Acal$ is
			closed under joins of arbitrary chains.
		\item if $\Acal$ has an artinian cogenerator, $\tors^d\Acal$ is
			closed under meets of arbitrary chains.
	\end{enumerate}
	In particular, the conclusion of (1) holds if $\Acal=\mod R$ for a
	right noetherian ring $R$ of $\gldim R\leq 2$, and both conclusions hold if
	$\Acal=\mod A$ for an artin algebra $A$ of $\gldim A\leq 2$.
\end{prop}

\begin{proof}
	We prove (1), as the argument for (2) is dual.
	Consider a chain $\tau_1\leq \tau_2\leq\cdots$ in $\tors^d\Acal$. If $M$ is a
	noetherian generator, this chain of torsion pairs induces a chain of torsion parts
	$t_1\subseteq t_2\subseteq\cdots \subseteq M$ of $M$, which must stabilise by
	assumption. If $N\geq 1$ is such that $t_i=t_N$ for every $i\geq N$ and we
	write $\epsilon\colon 0\to t_N\to M\to f_N\to 0$ for the corresponding torsion
	sequence, this means that $\tau_i\in\tors(M/\epsilon)$ for every $i\geq N$.
	Since $\tors(M/\epsilon)$ is an interval, it is closed under arbitrary joins,
	so it contains $\bigvee_{i\geq N}\tau_i$.
	We conclude because $\tors^d(\Acal)\cap \tors(M/\epsilon)$ is an upper-set
	in $\tors(M/\epsilon)$ by Corollary~\ref{cor:upper-lower} and it contains $\tau_N\leq \bigvee_{i\geq
	N}\tau_i$.
\end{proof}

\subsection{Finiteness of the partition $\tors(M/\epsilon)$}

Motivated by Corollary~\ref{cor:derived-partition}, we ask the question of when
the partition of $\tors(\Acal)$ into intervals $\tors(M/\epsilon)$ induced by an
object $M$ is finite. Of course this can happen trivially, as
$\tors(\Acal)=[\mathbf{0},\mathbf{1}]$ is an interval itself, corresponding to
the choice $M=0$. We also mentioned in \S\ref{subsec:lattice} that this
partition consists of two classes if and only if $M$ is torsion-simple.

In view of the Corollary, however, one would like $M$ to be a generator
and a cogenerator, which would show that, under the hypothesis that $\Acal$ is a
length category with
$\gldim\Acal\leq 2$, the subset $\tors^d(\Acal)$ is a union of
finitely many intervals in $\tors(\Acal)$ (a strong property for a subset of a
lattice).

We collect a few examples of objects $M$ for which the partition of
$\tors(\Acal)$ is finite. Observe that the partition is indexed by the torsion
sequences $\epsilon$ for $M$, which in turn are naturally in bijection with the
set $\mathbf{t}M$ of possible torsion parts of $M$.

\begin{lemma}
	For a commutative noetherian ring $R$, the set $\mathbf{t}M$ is finite for
	every $M$ in $\mod R$.
\end{lemma}

\begin{proof}
	The torsion theory of $\mod R$ for a commutative noetherian ring is controlled
	by the prime spectrum $\Spec R$, via the notion of the set of \emph{associated
	primes} $\Ass M\subseteq\Spec R$ of a module $M\in\mod R$. This is a finite
	set. Every torsion pair in $\mod R$ is described by a set $V\subseteq\Spec R$,
	in such a way that an element $m\in M$ lies in the torsion part of $M$ if
	and only if $\Ass(mR)\subseteq\Ass(M)$ is contained in $V$. Therefore,
	$\mathbf{t}M$ is in bijection with some of the finitely many subsets of
	$\Ass(M)$.
\end{proof}

Returning to length categories, we have the following.

\begin{lemma}\label{lemma:c-in-a}
		Let $\Acal$ be a length category, and let $\Ccal\subseteq\Acal$ be a class
		of objects such that:
		\begin{enumerate}[label=(\alph*)]
			\item	$\Ccal$ is closed under subobjects;
			\item every object of $\Ccal$ can be written as a finite direct sum of objects
			from a finite set $\Scal$.
		\end{enumerate}
		Then, for every object $M$ of $\Ccal$ the set of isomorphism classes of
		subobjects of $M$ is finite.
\end{lemma}

\begin{proof}
	Every subobject of $M$ lies in $\Ccal$ by (a), so it is determined up to
	isomorphism by the exponents of the objects of $\Scal$ in its decomposition
	given by (b). There are only finitely many choices for these exponents, as
	$\Scal$ is finite and the total length is bound by that of $M$.
\end{proof}

Examples of such classes $\Ccal\subseteq\Acal$ are: any class of semisimple
objects $\add \Scal\subseteq\Acal$, when $\Scal$ is a finite set of simple
objects of $\Acal$; the class $\proj A\subseteq\mod A$ of projectives over a
hereditary artin algebra $A$. Dually, the class $\inj A\subseteq\mod A$ of
injectives over a hereditary artin algebra $A$ satisfies the dual properties.
We are interested in these classes because of the following.

\begin{lemma}
	Whenever $\Ccal\subseteq\Acal$ is a class satisfying the conditions of
	Lemma~\ref{lemma:c-in-a} or their duals, the set $\mathbf{t}M$ is finite for
	every $M$ in $\Ccal$.
\end{lemma}

\begin{proof}
	We prove the claim for $\Ccal\subseteq\Acal$ as in Lemma~\ref{lemma:c-in-a}.
	Let $M\in \Ccal$, so that it only has finitely many isomorphism classes of
	subobjects. A subobject of $M$ which is isomorphic to a torsion
	subobject will be itself torsion; therefore, a torsion part of $M$ must be the
	only subobject in its own isomorphism class. Hence there is at most one for
	every isomorphism class of subobjects of $M$, and we conclude.

	If $\Ccal$ satisfies the dual properties, the objects of $M$ will have
	finitely many quotients up to isomorphisms, and one argues using the
	torsion-free parts.
\end{proof}

\begin{rmk}
	We remark that in general it is possible, even for a (nonprojective and
	noninjective) finitely generated module
	over a hereditary finite-dimensional algebra, to have infinitely many
	nonisomorphic submodules. For example, over the $3$-Kronecker algebra
	$k(\begin{tikzcd}[sep=10pt,cramped]\bullet \ar[shift left=2pt]{r}\ar{r}\ar[shift
	right=2pt]{r} & \bullet \end{tikzcd})$, consider the \emph{tree module of type
	$(2,2)$} (in the notation of \cite{ring-16}):
	\[\begin{tikzcd}
		x \ar{d}{\alpha} & y \ar{dl}{\beta} \ar{d}{\gamma} \\
		\alpha x=\beta y & \gamma y
	\end{tikzcd}\]
	where $x,y$ are generators and $\alpha, \beta, \gamma$ denote the three
	scalars corresponding to the arrows of the quiver. Then the 
	elements $\lambda x-y$, for any scalar $\lambda\in k$, are annihilated by
	$\alpha+\lambda\beta$, and therefore they all generate nonisomophic cyclic
	submodules.
	Observe, however, that these submodules are not torsion parts, as
	they all have an epimorphism to the corresponding quotient, which is always
	isomorphic to the simple injective module. Therefore, we still do not have an
	answer for the following question, even stated in this strong form.
\end{rmk}

\begin{question}
	Is the set $\mathbf{t}M$ finite for every finitely generated module $M$ over an
	artin algebra?
\end{question}

Turning back to our partition of $\tors(A)$, we deduce the following result.

\begin{prop}\label{prop:hereditary-partition-finite}
	If $A$ is a hereditary artin algebra, the partition of $\tors(A)$ given by the
	intervals $\tors((A\oplus DA)/\epsilon)$ is finite. In particular, the subset
	$\tors^d(A)$ is the union of finitely many intervals in $\tors(A)$.
\end{prop}

\begin{proof}
	As mentioned, the classes $\proj A$ and $\inj A$ satisfy the conditions of
	Lemma~\ref{lemma:c-in-a} and their duals, respectively. Therefore $\mathbf{t}A$ and
	$\mathbf{t}DA$ are finite, and thus so is $\mathbf{t}(A\oplus DA)$. The last
	claim follows from Corollary~\ref{cor:derived-partition}.
\end{proof}

This result is of course of special interest when $\tors(A)$ is infinite (i.e.\
when $A$ is not \emph{$\tau$-tilting finite}, see \cite{demo-iyam-jass-19}, and
\S\ref{subsec:2tilting} later on), see 
Figure~\ref{fig:kronecker} for an illustration.
We remark that for finite-dimensional algebras we have the following characterisation.

\begin{prop}
	Let $\Lambda$ be a hereditary finite-dimensional algebra. Then $\tors(\Lambda)$ is finite
	if and only if $\Lambda$ is representation-finite.
\end{prop}

\begin{proof}
	Clearly, if $\Lambda$ is representation-finite, it is also $\tau$-tilting
	finite. For the converse, assume $\Lambda$ is representation-infinite. By
	\cite[Theorem~3.8]{demo-iyam-jass-19}, to show that
	that it is not $\tau$-tilting finite, we need to construct a torsion pair that
	is not \emph{functorially finite}, that is, whose torsion-free class is not
	covering. The (additive closure of the) preprojective component
	$\Pcal(\Lambda)$ is the torsion-free class of a torsion pair (see for example
	\cite[Cor.~2.13]{asse-sims-skow-06}). However, it is not a covering class.
	Indeed, $\Pcal(\Lambda)$ contains infinitely many indecomposable modules, all
	with a nonzero morphism to the injective cogenerator $D\Lambda$, but for every
	indecomposable module $P$ in $\Proj(\Lambda)$, the set $\Hom(Q,P)$ is zero for
	all but finitely many indecomposable $\Lambda$-modules
	\cite[Lemma~VII.2.5]{asse-sims-skow-06}. Therefore, no module in
	$\Pcal(\Lambda)$ can be a $\Pcal(\Lambda)$-cover of $D\Lambda$.
\end{proof}

\begin{figure}
\usetikzlibrary{decorations.text, fit, positioning, quotes, backgrounds,calc}
\begin{tikzpicture}
	\node (At) at (0,5) {$\bullet$};
	\node (Bt) at (1,4.3) {$\bullet$};
	\node (Ct) at (2,3.3) {$\bullet$};
	\node (Dt) at (3,1.5) {$\bullet$};
	\node (Et) at (2,.7) {$\bullet$};
	\node (Ft) at (4,.7) {$\bullet$};
	\node (Ab) at (0,-5) {$\bullet$};
	\node (Bb) at (1,-4.3) {$\bullet$};
	\node (Cb) at (2,-3.3) {$\bullet$};
	\node (Db) at (3,-1.5) {$\bullet$};
	\node (Eb) at (2,-.7) {$\bullet$};
	\node (Fb) at (4,-.7) {$\bullet$};

	\node (S) at (-3,0) {$\bullet$};

	\draw[-] (At) -- (Bt);
	\draw[-] (Bt) -- (Ct);
	\path (Ct) --  node[sloped]{$\cdots$} (Dt);
	\draw[-] (Ab) -- (Bb);
	\draw[-] (Bb) -- (Cb);
	\path (Cb) --  node[sloped]{$\cdots$} (Db);

	\draw (Et) -- (Dt) -- (Ft);
	\draw (Eb) -- (Db) -- (Fb);

	\draw (At) -- (S) -- (Ab);

	\path (Et) -- node[sloped]{$\cdots$} (Eb);
	\path (Et) -- node[sloped]{$\cdots$} (Ft);
	\path (Et) -- node[sloped,near start]{$\cdots$} node[sloped,near end]{$\cdots$} (Fb);
	\path (Ft) -- node[sloped]{$\cdots$} (Fb);
	\path (Eb) -- node[sloped]{$\cdots$} (Fb);
	\path (Ft) -- node[sloped,near start]{$\cdots$} node[sloped,near end]{$\cdots$} (Eb);
	
	\begin{scope}[on background layer]
		\fill[fill=lightgray] (At) circle [radius=12pt];
		\fill[fill=lightgray] (Bt) circle [radius=12pt];
		\fill[fill=lightgray] (Ab) circle [radius=12pt];
		\fill[fill=lightgray] (Bb) circle [radius=12pt];
		\draw (S) circle [radius=12pt];
		\fill[fill=lightgray,rounded corners=24pt] (1.5,0) -- ($(Cb.center)+(250:24pt)$)
			.. controls (5.5,0) ..  ($(Ct.center)+(110:24pt)$) -- cycle;
	\end{scope}
\end{tikzpicture}
	\caption{The Hasse quiver of the lattice $\tors(\Lambda)$, for the Kronecker
	algebra $\Lambda=k(\bullet\rightrightarrows\bullet)$ (see
	e.g.\ \cite[Ex.~1.3]{thom-21} for an explicit description of the lattice
	points). The finitely many
	intervals of the partition given by the sets $\tors((\Lambda\oplus
	D\Lambda)/\epsilon)$ are highlighted, with the ones lying in
	$\tors^d(\Lambda)$ shaded in gray. To obtain this, one can inspect the actual
	torsion pairs, which shows that, except for the one in the center-left,
	they all split, and therefore induce derived equivalence. The only one
	remaining fails the criterion of Theorem~\ref{thm:chz}, and also the criterion
	given by Corollary~\ref{cor:path-algebra} later on.
	}
	\label{fig:kronecker}
\end{figure}

\subsection{A necessary condition for derived equivalence}

From Theorem~\ref{thm:chz}, we also deduce a necessary condition for a torsion
pair in an abelian category $\Acal$ to induce derived equivalence.

\begin{cor}\label{cor:necessary}
	If $\tau=(\Tcal,\Fcal)$ induces derived equivalence, we must have:
	\[\Fcal^\bot\subseteq
	\sub\Tcal\quad\text{and}\quad{}^\bot\Tcal\subseteq\fac\Fcal.\]
	In particular, since both $\Tcal$ and $\Fcal$ are closed under direct
	summands, we have:
	\[\Fcal^\bot\cap\Inj\Acal\subseteq
	\Tcal={}^\bot\Fcal\quad\text{and}\quad
	{}^\bot\Tcal\cap\Proj\Acal\subseteq\Fcal=\Tcal^\bot.\]
\end{cor}

\begin{proof}
	Follows immediately from the existing of sequences $0\to f_0\to f_1\to a\to
	t_0\to t_1\to 0$ as in Theorem~\ref{thm:chz}.
\end{proof}

This necessary condition will be a key ingredient later on. For now, we observe
that its two halves become sufficient conditions for torsion
pairs in a TTF-triple.

\begin{prop}[{\cite[Prop.~4.3]{chen-han-zhou-19}}]\label{prop:ttf}
	Let $(\Tcal,\Fcal)$ be a torsion pair in $\Acal$.
	\begin{enumerate}
		\item Assume that $\Fcal$ is a TTF-class. Then $(\Tcal,\Fcal)$ induces
			derived equivalence if and only if $\Fcal^\bot\subseteq \sub\Tcal$. If
			$\Acal$ has injective envelopes, this is equivalent to
			$\Fcal^\bot\cap\Inj\Acal\subseteq\Tcal={}^\bot\Fcal$.
		\item Assume that $\Tcal$ is a TTF-class. Then $(\Tcal,\Fcal)$ induces
			derived equivalence if and only if ${}^\bot\Tcal\subseteq \fac\Fcal$. If
			$\Acal$ has projective covers, this is equivalent to
			${}^\bot\Tcal\cap\Proj\Acal\subseteq\Fcal=\Tcal^\bot$.
	\end{enumerate}
\end{prop}

\begin{proof}
	The first part of each statement was observed in the paper by Chen, Han and
	Zhou. For the second part of (1), notice that if $\Fcal$ is a TTF-class, 
	the torsion pair $(\Fcal,\Fcal^\bot)$ is hereditary, so $\Fcal^\bot$ is
	closed under taking injective envelopes. Therefore, if $\Acal$ has injective
	envelopes, we have $\Fcal^\bot=\sub(\Fcal^\bot\cap \Inj\Acal)$. Now, since
	$\Tcal$ is closed under direct summands, we deduce that:
	\[ \Fcal^\bot\subseteq\sub\Tcal \iff \Fcal^\bot\cap\Inj\Acal\subseteq
	\sub\Tcal \iff \Fcal^\bot\cap\Inj\Acal\subseteq\Tcal.\]
	The second part of (2) is proved dually.
\end{proof}

\section{Applications of the CHZ criterion to artin algebras}
\label{sec:artin-applications}

\subsection{Third application: (co)hereditary torsion pairs over artin algebras}
\label{subsec:cohereditary}

Let $\Lambda$ be an artin algebra, and let $\simp\Lambda$ denote the set of
isoclasses of simple $\Lambda$-modules. For a module $M\in\mod(\Lambda)$, we
denote by $\supp(M)\subseteq\simp\Lambda$ the set of its composition factors,
and by $E(M)$ and $P(M)$ its injective envelope and projective cover,
respectively.

\begin{dfn}
	We define two functions on the set $\Pcal(\simp\Lambda)$ of subsets of $\simp\Lambda$:
	\begin{align*}
		\Phi_-&\colon \Pcal(\simp\Lambda)\to \Pcal(\simp\Lambda),
		\quad \Phi_-(\{\,S_1,\dots,S_r\,\})=\supp(\top E(S_1\oplus\cdots\oplus S_r))
		\\
		\Phi_+&\colon \Pcal(\simp\Lambda)\to \Pcal(\simp\Lambda),
		\quad \Phi_+(\{\,S_1,\dots,S_r\,\})=\supp(\soc P(S_1\oplus\cdots\oplus S_r)).
	\end{align*}
\end{dfn}

Before explaining how to use these functions to study derived equivalences, we
give a few examples to develop some intuition.

\begin{ex}
	If $\Lambda$ is hereditary, the top of an injective is itself injective,
	and the socle of a projective is itself projective.
	Therefore, the image of $\Phi_-$ consists of subsets of simple injectives, and
	the image of $\Phi_+$ of subsets of simple projectives.
\end{ex}

\begin{ex}\label{ex:self-injective}
	If $\Lambda$ is a self-injective algebra, then the indecomposable
	projectives are also indecomposable injectives, and therefore they have a
	simple socle. This yields the \emph{Nakayama permutation} $\nu$ on
	$\simp\Lambda$, defined by $S\mapsto \soc P(S)$, with inverse $\nu^{-1}\colon
	S\mapsto \top E(S)$. In this setting, our functions are then $\Phi_-=\nu^{-1}_\ast$
	and $\Phi_+=\nu_\ast$, so they are inverse bijections.

	In particular, the algebra $\Lambda$ is said to be \emph{weakly symmetric}
	(originally introduced in \cite{naka-nesb-38}, see
	\cite[p.~378]{skow-yama-11}) if the Nakayama permutation $\nu$ is the identity. In this case,
	$\Phi_-$ and $\Phi_+$ are the identity as well.
\end{ex}

\begin{ex}\label{ex:path-algebra}
	Let $\Lambda=kQ/I$ be the path algebra of a finite quiver $Q=(Q_0,Q_1)$
	with admissible ideal of relations $I$. We identify $\simp\Lambda$ with the
	set of vertices $Q_0$, by the natural assignment $i\mapsto S_i=\top
	(e_i\Lambda)$. In this setting the functions $\Phi_+, \Phi_-$ admit a
	clear combinatorial description.

	Say that a path $p\colon i\leadsto j$ in $Q$ is \emph{tail-maximal} with
	respect to $I$ if $p\notin I$ and $pa\in I$ for every arrow $a\colon j\to j'$
	in $Q_1$. In other
	words, a tail-maximal path represents a nonzero element of $\Lambda$ which is
	annihilated from the right by the radical of $\Lambda$, that is,
	an element of the right socle of $\Lambda$. A
	\emph{head-maximal} path is a tail-maximal path on the opposite quiver of $Q$.
	From this description, it is clear that for every vertex $i\in Q_0$, if
	$P_i=e_i\Lambda$ is the projective cover of $S_i$, the elements of $\soc P_i$
	are represented by tail-maximal paths starting in $i$. Since $j\in Q_0$
	belongs to $\supp(M)$ if and only if $Me_j\neq 0$, we obtain:
	\begin{align*}
		\Phi_+(\{i\})&=\{\,j\in Q_0\mid \exists\, p\colon i\leadsto j\text{
			path in }Q\text{ tail-maximal for }I\,\},\text{ and dually}\\
		\Phi_-(\{i\})&=\{\,j\in Q_0\mid \exists\, p\colon j\leadsto i\text{
			path in }Q\text{ head-maximal for }I\,\}.
	\end{align*}
	Given that $\Phi_-$ and $\Phi_+$ clearly preserve unions, this gives a full
	combinatorial description of the two functions. See Figure~\ref{fig:phi+} for an
	example.
\end{ex}

\begin{figure}[t]
	\[\begin{tikzcd}[sep=small]
		& & \bullet \ar{d} & & & & & & & & \circ \ar{d} \\
		\circ & \ar{l} \circ & \ar{l} \circ \ar[""{name=S, above}]{r} &
			\circ \ar[""{name=E,above}]{r} & \circ
		\ar[-,dashed,from=S,to=E,bend left=4em]
		& \; \ar[|->]{rr}{\Phi_+} && \; &
		\bullet & \ar{l} \circ & \ar{l} \circ \ar[""{name=S, above}]{r} &
			\bullet \ar[""{name=E,above}]{r} & \circ
		\ar[-,dashed,from=S,to=E,bend left=4em]
	\end{tikzcd}
	\]
	\[\begin{tikzcd}[sep=small]
		& & \circ \ar{d} & & & & & & & & \circ \ar{d} \\
		\circ & \ar{l} \circ & \ar{l} \circ \ar[""{name=S, above}]{r} &
			\bullet \ar[""{name=E,above}]{r} & \circ
		\ar[-,dashed,from=S,to=E,bend left=4em]
		& \; \ar[|->]{rr}{\Phi_+} && \; &
		\circ & \ar{l} \circ & \ar{l} \circ \ar[""{name=S, above}]{r} &
			\circ \ar[""{name=E,above}]{r} & \bullet
		\ar[-,dashed,from=S,to=E,bend left=4em]
	\end{tikzcd}
	\]
	\[
	\begin{tikzcd}[sep=small]
		& & \bullet \ar{d} & & & & & & & & \circ \ar{d} \\
		\bullet & \ar{l} \bullet & \ar{l} \circ \ar[""{name=S, above}]{r} &
			\bullet \ar[""{name=E,above}]{r} & \bullet
		\ar[-,dashed,from=S,to=E,bend left=4em]
		& \; \ar[|->]{rr}{\Phi_+} && \; &
		\bullet & \ar{l} \circ & \ar{l} \circ \ar[""{name=S, above}]{r} &
			\bullet \ar[""{name=E,above}]{r} & \bullet
		\ar[-,dashed,from=S,to=E,bend left=4em]
	\end{tikzcd}
	\]
	\caption{
		We depict a set of vertices of a quiver with relations by marking its
		elements in black. The picture illustrates the action of $\Phi_+$ on three such
		subsets. The map $\Phi_+$ moves vertices downstream of the arrows as far as
		it can without going across a relation.
	}\label{fig:phi+}
\end{figure}

For a set $\Scal\subseteq\simp\Lambda$, denote by
$\Scal^c:=\simp\Lambda\setminus\Scal$ the complement. The class
$\filt\Scal\subseteq\mod(\Lambda)$ of
modules filtered by simples in $\Scal$ is a TTF-class, by closure properties.
Observe that for a module $M\in\mod(\Lambda)$, we have that $M$ lies in
$(\filt\Scal)^\bot=\Scal^\bot$ if and only if $\supp(\soc M)\subseteq\Scal^c$,
and it lies in ${}^\bot(\filt\Scal)={}^\bot\Scal$ if and only if $\supp(\top
M)\subseteq\Scal^\bot$.

We deduce the following result about hereditary and cohereditary
torsion pairs in $\mod(\Lambda)$.

\begin{prop}\label{prop:phi-criterion}
	Let $\Lambda$ be an artin algebra and $\Scal\subseteq\simp\Lambda$ be a set of
	simples. Then in $\mod(\Lambda)$:
	\begin{enumerate}
		\item the torsion pair $({}^\bot\Scal,\filt\Scal)$ induces derived equivalence if and only if
			$\Phi_-(\Scal^c)\subseteq\Scal^c$;
		\item the torsion pair $(\filt\Scal,\Scal^\bot)$
			induces derived equivalence if and only if
			$\Phi_+(\Scal^c)\subseteq\Scal^c$.
	\end{enumerate}
\end{prop}

\begin{proof}
	We prove item (1), the other being dual. By item (1) of
	Proposition~\ref{prop:ttf}, the pair $({}^\bot\Scal,\filt\Scal)$ gives derived
	equivalence if and only if we have an inclusion
	$\Scal^\bot\cap\Inj(\mod(\Lambda))\subseteq{}^\bot\Scal$.

	Let $E$ be an injective module.
	As mentioned, $E$ lies in $\Scal^\bot$ if and only if $\supp(\soc E)$
	is contained in $\Scal^c$, and it lies in ${}^\bot\Scal$ if and
	only if $\supp(\top E)=\supp(\top E(\soc E))=\Phi_-(\supp(\soc E))$ is contained
	in $\Scal^c$. Therefore, we have $\Phi_-(\Scal^c)\subseteq\Scal^c$ if and only
	if $\Scal^\bot\cap\Inj(\mod(\Lambda))\subseteq {}^\bot\Scal$.
\end{proof}

Combining with Example~\ref{ex:path-algebra}, we obtain the following
combinatorial criterion.

\begin{cor}\label{cor:path-algebra}
	Let $\Lambda=kQ/I$ be a path algebra with relations, $S\subseteq Q_0$ a set of
	vertices and $\Scal=\{\,S_i\mid i \in S\,\}$ the corresponding set of simples.
	Then:
	\begin{enumerate}
		\item the torsion pair $({}^\bot\Scal,\filt\Scal)$ induces derived
			equivalence if and only if every nonzero path ending outside $S$ can be prolonged to
			a nonzero path starting outside $S$, that is:
			\[\forall\,p\colon i\leadsto j\text{ with }j\in S^c\text{ and }p\notin
			I\quad
			\exists\, q\colon j'\leadsto i\text{ such that }j'\in S^c\text{ and }qp\notin
			I.\]
		\item the torsion pair  $(\filt\Scal,\Scal^\bot)$ induces derived
			equivalence if and only if every nonzero path starting outside $S$ can be prolonged to
			a nonzero path ending outside $S$, that is:
			\[\forall\,p\colon j\leadsto i\text{ with }j\in S^c\text{ and }p\notin
			I\quad
			\exists\, q\colon i\leadsto j'\text{ such that }j'\in S^c\text{ and }pq\notin
			I.\]
	\end{enumerate}
\end{cor}

\begin{ex}
	In Figure~\ref{fig:phi+}, the third set considered is closed under $\Phi_+$,
	the others are not. Therefore, the hereditary torsion pairs supported on the
	complement of these sets give derived equivalence in the third case but not in
	the first two.
\end{ex}

\begin{ex}
	Let $\Lambda=kQ$ be the path algebra of an acyclic quiver, without relations. Let $S\subseteq Q_0$ be
	a set of vertices, and $\Scal$ the corresponding set of simple modules, as
	above. Then:
	\begin{enumerate}
		\item the cohereditary torsion pair $({}^\bot\Scal,\filt\Scal)$ induces
			derived equivalence if and only if whenever $S$ contains a source $i$, it
			also contains every vertex $j$ with a path $i\leadsto j$;
		\item the hereditary torsion pair $(\filt\Scal,\Scal^\bot)$ induces
			derived equivalence if and only if whenever $S$ contains a sink $j$, it
			also contains every vertex $i$ with a path $i\leadsto j$.
	\end{enumerate}
	This allows to easily compute the number of (co)hereditary torsion pairs
	inducing derived equivalence for the path algebra of a given acyclic quiver
	$Q$. For example, let $Q$ be a quiver with $n$ vertices whose underlying
	graph has a star shape:
	
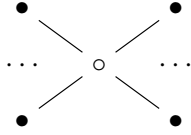
\[\begin{tikzcd}[sep=small,cramped]
		\bullet \ar[-]{dr} && \bullet \\
		\cdots & \circ \ar[-]{ur} \ar[-]{dr} & \cdots \\
		\bullet \ar[-]{ur} && \bullet
	\end{tikzcd}\]
	Let $1\leq k\leq n-1$ be the number of sinks. Since every non-sink has a path
	to every sink, out of the $2^n$ hereditary torsion pairs there are
	$2^{n-k}+(2^k-1)$ inducing derived equivalence (one for every set not
	containing any sink, plus one for any nonempty set of sinks). As a function of
	$k$, this has a symmetry, reaching a maximum of $2^{n-1}+1$ when either $k=1$ or
	$k=n-1$, and a minimum when the number of sinks is closest to half the number of
	vertices, with value $2^{\frac{n}2+1}-1$ for $n$ even and
	$3\cdot2^{\lfloor\frac{n}2\rfloor}-1$ for $n$ odd.

	We remark that from the Corollary it follows that it is possible to add and
	delete arrows from a quiver without affecting the number of hereditary torsion
	pairs inducing derived equivalence, as long as this operation does not change
	the fact that there is or is not a path from a given vertex to a given sink.
	For example, if in $Q$ all paths to every sink go through a given vertex
	$\circ$, it is possible to reduce $Q$ to a star-shaped graph as above, by
	making every non-sink other than $\circ$ into a source with an arrow to
	$\circ$. To illustrate this, the following two quivers have the same number of
	hereditary torsion pairs giving derived equivalence:
	
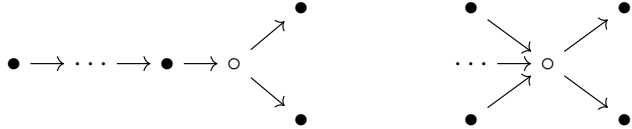
\[
		\begin{tikzcd}[sep=small,cramped]
		& & & & \bullet \\
		\bullet \ar{r} & \cdots \ar{r} & \bullet \ar{r} & \circ \ar{ur}\ar{dr}\\
		& & & & \bullet
	\end{tikzcd}
		\hspace{5em}
	\begin{tikzcd}[sep=small,cramped]
		\bullet \ar{dr} & & \bullet \\
		\cdots \ar{r} & \circ \ar{ur}\ar{dr} & \phantom{\cdots}\\
		\bullet \ar{ur} & & \bullet
	\end{tikzcd}
\]

	The previous calculation for star-shaped graphs covers all the
	possible orientations of the Dynkin diagram $\Abb_3$, yielding that for all the
	corresponding algebras, $5$ out of the $8$ hereditary torsion pairs induce
	derived equivalence. For $\Abb_4$, however, this number depends on the
	orientation, as it is $8$ out of $16$ for the two orientations with alternating
	arrows (which cannot be reduced to the star-shaped case), and $9$ out of $16$
	for the others. The quiver $\Dbb_4$ reaches an even lower number, with all
	orientations having $9$ out of $16$ hereditary
	torsion pairs giving derived equivalence, except for the orientation with one
	source and two sinks, which has $7$, the minimum we have computed for
	a star-shaped quiver with $4$ vertices.
\end{ex}

\begin{ex}
	Let $\Lambda=kQ/I$, for $I=\rad^2(kQ)$, be a radical-square-zero algebra. Then, in the
	notation of Corollary~\ref{cor:path-algebra}(2), we will never have $pq\notin
	I$, by construction. Therefore, the hereditary torsion pair corresponding to a
	set $S$ of vertices induces derived equivalence if and only if whenever there
	is a path $p\colon j\leadsto i$ with $i\in S$, we also have $j\in S$. Indeed, $S$ must
	be closed under predecessors along single arrows (which do not belong to $I$),
	and we conclude by induction on the length of $p$.

	For example, when $Q$ is the linearly oriented $\Abb_n$ quiver, there are
	exactly $n+1$ hereditary torsion pairs in $\mod(k\Abb_n/\rad^2)$ inducing
	derived equivalence (one for each initial set of vertices, including the empty
	set).
\end{ex}

Using Example~\ref{ex:self-injective}, we also obtain a criterion for
self-injective algebras. Observe that the part about weakly symmetric algebras also
follows from the more general fact that if $\Lambda$ is weakly symmetric, all
functorially finite torsion pairs in $\mod(\Lambda)$ (which includes the
(co)hereditary ones) induce derived equivalence by
\cite[Prop.~3.6]{augu-duga-21}. The part about the case of a transitive Nakayama
permutation can also be deduced from (the proof of)
\cite[Cor.~5.8]{xi-zhang-25}.

\begin{cor}
	Let $\Lambda$ be a self-injective artin algebra with Nakayama permutation
	$\nu$, and $\Scal\subseteq\simp\Lambda$ a set of simples. Then each of
	the torsion pairs $({}^\bot\Scal,\filt\Scal)$, $(\filt\Scal,\Scal^\bot)$ 
	$({}^\bot\Scal^c,\filt\Scal^c)$ and
	$(\filt\Scal^c,(\Scal^c)^\bot)$ induces derived equivalence if and only if $\Scal$
	is closed under $\nu$. In particular, this is always the case if $\Lambda$ is
	weakly symmetric, and never the case if $\nu$ is transitive and $\emptyset\neq
	\Scal\subsetneq\simp\Lambda$.
\end{cor}

\begin{proof}
	Observe that since $\nu$ is a permutation of a finite set $\simp\Lambda$, the
	conditions
	that $\Scal^c$ is closed under $\nu^{-1}$, that $\Scal$
	is closed under $\nu$, that $\Scal$ is closed under $\nu^{-1}$ and that
	$\Scal^c$ is closed under $\nu$ are all equivalent. By
	Example~\ref{ex:self-injective}, they mean that $\Scal^c$ and $\Scal$ are
	closed under $\Phi_-$ and $\Phi_+$, and we conclude by
	Proposition~\ref{prop:phi-criterion}.
\end{proof}

\begin{ex}\label{ex:self-injective-two-vertices}
	Consider the quiver $Q=(\begin{tikzcd}[sep=small,cramped]\bullet\ar[bend
	left]{r}&\ar[bend left]{l}\bullet\end{tikzcd})$, and the self-injective algebras
	$\Lambda_n:=kQ/\rad^n$, for $n\geq 2$. Over these algebras there
	are always six torsion pair, which are all either hereditary or cohereditary; moreover,
	the lattice $\tors(\Lambda_n)$ is independent of $n$. The
	Nakayama permutation of $\Lambda_n$ exchanges the two simples, if $n$ is even,
	while it is the identity if $n$ is odd. Hence, drawing the
	Hasse quivers of $\tors(\Lambda_n)$ and highlighting in black the torsion
	pairs inducing derived equivalence, we obtain two possible pictures:
	\[
		\arraycolsep=2pc
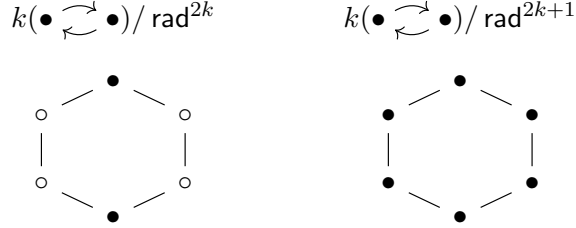
		\begin{array}{cc}
			k(\begin{tikzcd}[sep=small,cramped]\bullet\ar[bend left]{r}&\ar[bend
			left]{l}\bullet\end{tikzcd})/\rad^{2k} &
			k(\begin{tikzcd}[sep=small,cramped]\bullet\ar[bend left]{r}&\ar[bend
			left]{l}\bullet\end{tikzcd})/\rad^{2k+1} \\[10pt]
		\begin{tikzcd}[row sep=0,column sep=1pc]
			& \bullet \ar[-]{dl}\ar[-]{dr} \\
			\circ \ar[-]{dd} & & \circ \ar[-]{dd} \\
			\phantom{\bullet}\\
			\circ \ar[-]{dr} & & \circ \ar[-]{dl} \\
			& \bullet
		\end{tikzcd}
			&
		\begin{tikzcd}[row sep=0,column sep=1pc]
			& \bullet \ar[-]{dl}\ar[-]{dr} \\
			\bullet \ar[-]{dd} & & \bullet \ar[-]{dd} \\
			\phantom{\bullet}\\
			\bullet \ar[-]{dr} & & \bullet \ar[-]{dl} \\
			& \bullet
		\end{tikzcd}
		\end{array}
	\]
\end{ex}

\subsection{Fourth application: 2-tilting connectedness}
\label{subsec:2tilting}

For this application, let $\Lambda$ be a finite-dimensional algebra over a
field $k$.

Recall that a torsion pair $(\Tcal,\Fcal)$ in $\mod(\Lambda)$ is
\emph{functorially finite} if $\Tcal$ is an enveloping class (equivalently,
if $\Fcal$ is a covering class), see \cite[Prop.~1.1]{adac-iyam-reit-14} and
references therein. This is also equivalent to the fact
that $\Tcal=\gen M$, for a \emph{support $\tau$-tilting} module $M$
\cite[Thm.~2.7]{adac-iyam-reit-14}.
Denoting by $\ftors(\Lambda)\subseteq\tors(\Lambda)$ the set of functorially
finite torsion pairs of $\mod(\Lambda)$, we recall the following remarkable
result of Demonet, Iyama and Jasso.

\begin{thm}[{\cite[Thm.~3.8]{demo-iyam-jass-19}}]\label{thm:tau-tilting-finite}
	For a finite-dimensional algebra $\Lambda$, the following are equivalent:
	\begin{enumerate}
		\item $\ftors(\Lambda)$ is finite;
		\item $\ftors(\Lambda)=\tors(\Lambda)$.
	\end{enumerate}
	In this case we say that $\Lambda$ is \emph{$\tau$-tilting finite}.
\end{thm}

Functorially finite torsion pairs are also in bijection with \emph{two-term
silting complexes} \cite[Thms.~2.7 and~3.2]{adac-iyam-reit-14}. This bijection
preserves a lot of information, of which we give two examples, to motivate the
results of this subsection. First, on silting complexes there is an operation of
\emph{irreducible mutation} \cite{aiha-iyam-12}, and two two-term silting
complexes are linked by irreducible mutation if and only if there is a minimal
inclusion between the corresponding torsion pairs. Second, among the two-term
silting complexes there are the two-term \emph{tilting} complexes, which
correspond to the functorially finite torsion pairs inducing derived equivalence
\cite[Prop.~5.1]{psar-vito-18}.

For a functorially finite torsion pair $\tau=(\Tcal,\Fcal)\in\ftors(\Lambda)$ in
$\mod(\Lambda)$, denote by $\Hcal_\tau$ the corresponding HRS-tilt of
$\mod(\Lambda)$ in $\D^b(\Lambda)$, and by $T$ the associated two-term silting
complex, as above. Then $\Hcal_\tau$ is not only an abelian
category, but it is equivalent to the category $\mod(B)$ of modules over the
endomorphism algebra $B:=\End(T)$ \cite[Lemma~5.3]{koen-yang-14}.

Now consider the following examples. These
two algebras have finite representation type and have few indecomposable
representations, up to isomorphism, which allows to easily compute the lattices of torsion
pairs. The subset of pairs inducing derived equivalence can be determined using
Theorem~\ref{thm:chz} (see also Example~\ref{ex:self-injective-two-vertices} for the second algebra).

\begin{ex}
	For two path algebras with relations, we draw the Hasse quiver of the lattice
	$\tors(\Lambda)$ with the subset $\tors^d(\Lambda)$ highlighted in black.
	\[
		\arraycolsep=2pc
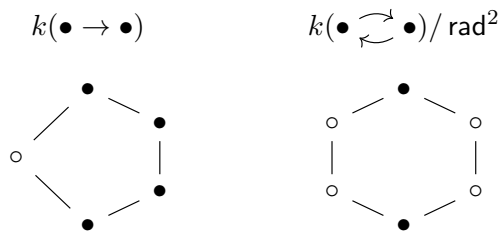
		\begin{array}{cc}
			k(\bullet\to\bullet) &
			k(\begin{tikzcd}[sep=small,cramped]\bullet\ar[bend left]{r}&\ar[bend
			left]{l}\bullet\end{tikzcd})/\rad^2 \\[10pt]
		\begin{tikzcd}[row sep=0,column sep=1pc]
			& \bullet \ar[-]{ddl}\ar[-]{dr} \\
			& & \bullet \ar[-]{dd} \\
			\circ \ar[-]{ddr} \\
			& & \bullet \ar[-]{dl} \\
			& \bullet
		\end{tikzcd}
			&
		\begin{tikzcd}[row sep=0,column sep=1pc]
			& \bullet \ar[-]{dl}\ar[-]{dr} \\
			\circ \ar[-]{dd} & & \circ \ar[-]{dd} \\
			\phantom{\bullet}\\
			\circ \ar[-]{dr} & & \circ \ar[-]{dl} \\
			& \bullet
		\end{tikzcd}
		\end{array}
	\]
\end{ex}

These two examples exhibit different behaviours: in the first,
$\tors^d(\Lambda)$ is connected in $\tors(\Lambda)$, in the second it is not.
More precisely, in the first case $\tors^d(\Lambda)$ is a union of maximal
chains in $\tors(\Lambda)$. In terms of two-term silting and tilting complexes,
in the first case irreducible silting mutation acts transitively on two-term tilting
complexes, while in the other it does not. In this subsection we give a
sufficient condition to ensure the first kind of behaviour.

Say that an algebra $\Lambda$ is \emph{acyclic} if there are no cyclic chains of
nonzero noninvertible morphisms $P_1\to P_2\to \cdots \to P_n\to P_1$ between indecomposable
projectives (or equivalently, injectives). The following lemma says that if
$\Lambda$ is acyclic, it is possible to make a first step from $\mathbf{0}$ and from
$\mathbf{1}$ in the direction of any element of $\tors^d\Lambda$, without
leaving this subset.

\begin{lemma}\label{lemma:acyclic}
	Let $\Lambda$ be an artin algebra, and assume that it is acyclic. Then for
	every torsion pair $(\Tcal,\Fcal)\in\tors^d\Lambda$, there exist simple
	modules $S_t\in\Tcal$ and $S_f\in\Fcal$ such that $(\filt S_t,S_t^\bot)\leq
	(\Tcal,\Fcal)\leq({}^\bot S_f,\filt S_f)$ belong to $\tors^d\Lambda$.
\end{lemma}

\begin{proof}
	We prove only the existence of $S_t$, as for $S_f$ the argument is
	similar.
	We are assuming that $(\Tcal,\Fcal)$ induces derived equivalence, so by
	Corollary~\ref{cor:necessary} we have that
	${}^\bot\Tcal\cap\proj\Lambda\subseteq \Fcal$. Now, assume by contradiction
	that for every simple $S\in\Tcal$, the hereditary torsion pair $(\filt
	S,S^\bot)$ does not induce derived equivalence. For every choice of a simple
	$S_i\in\Tcal$ the corresponding torsion pair must then fail the criterion of
	Proposition~\ref{prop:ttf}(2), so there must be an indecomposable
	projective module $P_i'$ such that $P_i'\in {}^\bot\filt S_i$ but $P_i'\notin
	S_i^\bot$. This in particular means that there is a nonzero morphism $S_i\to
	P_i'$, and therefore a nonzero morphism $P_i\to P_i'$ from the projective cover of
	$S_i$. Moreover, this morphism is not an isomorphism, as $P_i'\in{}^\bot\filt S_i$.

	Now, by construction $\Fcal\subseteq S_i^\bot$, and therefore $P_i'$ does not
	belong to $\Fcal$. Since we are assuming that $(\Tcal,\Fcal)$ induces derived
	equivalence, by the necessary condition of Corollary~\ref{cor:necessary} we
	must have that $P_i'$ does not belong to ${}^\bot\Tcal$ either. Hence there
	must be a nonzero morphism $P_i'\to T$, for $T\in\Tcal$. This morphism factors
	through the projective cover of $T$, which is the direct sum of the projective
	covers of the simple factors of $\top T\in\Tcal$. Therefore, we obtain at
	least one nonzero morphism $P_i'\to P_{\sigma(i)}$, where $P_{\sigma(i)}$ is
	the projective cover of one of the simples $S_{\sigma(i)}\in \Tcal$ appearing
	in $\top T$. This procedure defines a function $\sigma$ on the set of simple
	modules in $\Tcal$, and
	a chain of nonzero morphisms, of which at least those of the form
	$P_i\to P_i'$ are not isomorphisms:
	\[P_1\to P_1'\to P_{\sigma(1)}\to P_{\sigma(1)}'\to P_{\sigma^2(i)}\to \cdots\]
	Since $\Tcal$ only contains finitely many simples, we must have
	$\sigma^d(i)=i$ for some $i$ and $d\geq 0$, which yields a cycle of nonzero
	noninvertible morphisms between projectives. This is a contradiction.
\end{proof}

Our goal is to use this lemma as an inductive step to build a chain of torsion
pairs giving derived equivalence. To this end,
let us give the following definition.

\begin{dfn}
	An artin algebra $\Lambda$ is \emph{2-tilting acyclic} if the endomorphism
	algebra of any two-term tilting complex over $\Lambda$ is acyclic.
\end{dfn}

We first a couple of examples of algebras of this kind, and then proceed with
our Proposition.

\begin{ex}
	Recall that an algebra is \emph{piecewise-hereditary} if it is
	derived-equivalent to a hereditary algebra. Being piecewise-hereditary is
	obviously a derived-invariant property, and it implies being acyclic
	\cite[Lemma~IV.1.10]{happ-88}. 
	Therefore every piecewise-hereditary algebra is 2-tilting acyclic.
\end{ex}

\begin{ex}
	Let $\Lambda=k\Abb_n/I$ be an acyclic Nakayama algebra.
	For small values of $n$, all algebras of this form are piecewise-hereditary.
	This, however, is decidedly not true in general: see \cite{foss-oppe-stai-24}
	for counterexamples, and \cite[Table~1]{happ-seid-10} for a classification of
	the piecewise-hereditary Nakayama algebras of the special form
	$\Lambda(n,r):=k\Abb_n/\rad^r$.

	In the next Proposition we show that in $\K:=\K^{[-1,0]}(\proj\Lambda)$ there are no
	cyclic chains of nonzero, noninvertible morphisms between indecomposable
	objects. In particular, this implies that $\Lambda$ is 2-tilting acyclic, as
	for every two-term tilting complex $T\in\D^b(\Lambda)$ with endomorphism
	algebra $B:=\End(T)$ we have $\proj B\simeq \add T\subseteq \K$.
\end{ex}

\begin{prop}
	For an algebra of the form $\Lambda=k\Abb_n/I$, there are no cyclic chains of
	nonzero, noninvertible morphisms between indecomposable objects of
	$\K:=\K^{[-1,0]}(\proj\Lambda)$.
\end{prop}

\begin{proof}
	The linear order of $\Abb_n=(1\to
	2\to \cdots \to n)$
	induces a total order on the indecomposable projectives $\Lambda$-modules
	$P_1, \dots, P_n$, such that $\Hom_\Lambda(P_i,P_j)=0$ for $i<j$. When instead
	$i\geq j$, we have that $\Hom_\Lambda(P_i,P_j)\neq 0$ if and only if the
	unique path $j\leadsto i$ does not lie in $I$. In particular, this means that
	for every $i\geq j\geq k$, when $\Hom_\Lambda(P_i,P_k)\neq 0$ then any two
	nonzero morphisms $P_i\to P_j$ and $P_j\to P_k$ compose to a nonzero morphism.

	The indecomposable objects in $\K$ have one of three forms: they are either
	stalk complexes consisting of a single projective module, either in degree
	$-1$ or in degree $0$, or otherwise they are of the form $(P_i\to P_j)$,
	where $P_i$ and $P_j$ are nonzero projectives with $i\geq j$ and the differential
	is the unique nonzero morphism between them (up to scalars). Assume by
	contradiction that there is a cycle of nonzero, noninvertible morphisms
	between indecomposable objects of $\K$. Since $\Lambda$ is acyclic, this
	cycle cannot involve only stalk complexes in a single degree. Moreover, since
	there are no morphisms in either direction between stalk complexes of
	projectives in different degrees, this cycle must contain at least one
	two-term complex, as above. For the same reason, it is possible to section the
	cycle into a composition of paths $p$ of length at least $1$ with the following properties: (a) $p$
	starts and ends in two-term complexes, and (b) all the other objects appearing
	in $p$ are stalk complexes in the same degree. Such a path
	$p\colon (P_{i_0}\to P_{j_0})\leadsto (P_{i_1}\to P_{j_1})$ can be of one
	of three kinds, depending on whether it only consists of one morphism between
	two-term complexes, or otherwise on the degree of the stalk complexes in the middle. We
	are going to show that in each case we must have that $i_0\geq i_1$ and
	$j_0\geq j_1$, and at least one of these inequalities must be strict. This
	will then contradict the fact that the paths $p$ compose to a
	cycle.

	Case 1. We start from the case in which $p$ is just a nonzero, noninvertible morphism
	\[\begin{tikzcd}
		P_{i_0} \ar{d}{d_0} \ar{r}{f} & P_{i_1}\phantom{.} \ar{d}{d_1} \\
		P_{j_0} \ar{r}{g} & P_{j_1}.
	\end{tikzcd}\]
	At least one of $f$ and $g$ must be nonzero. We consider the three possible
	cases. If $f\neq 0=g$, then we have $i_0\geq i_1$. Moreover, by the two
	differentials we also have that $i_0\geq j_0$ and $i_1\geq j_1$. Assume we had
	$j_0<j_1$. Then we would have $i_0\geq i_1\geq j_1>j_0$, and since
	$d_0\in\Hom_\Lambda(P_{i_0},P_{j_0})\neq 0$, we would
	deduce that the composition $d_1f$ cannot vanish, by the observation
	above. This contradicts $g=0$. In a similar way, one excludes the case
	$f=0\neq g$, and therefore we conclude that $f\neq 0\neq g$, and $i_0\geq i_1$
	and $j_0\geq j_1$. If these inequalities were both equalities, then $f$ and
	$g$ would both be invertible, and the morphism of complexes an isomorphism.

	Case 2. Assume now that $p$ has the form:
	\[\begin{tikzcd}
		P_{i_0} \ar{d}{d_0} \ar{r}{f} & P_{k_1} \ar{d} \ar{r} & \cdots
		\ar{r} & P_{k_s} \ar{d} \ar{r}{g} & P_{i_1}\phantom{.} \ar{d}{d_1} \\
		P_{j_0} \ar{r} & 0 \ar{r} & \cdots \ar{r} & 0 \ar{r} & P_{j_1}.
	\end{tikzcd}\]
	where the morphisms in degrees $-1$ are necessarily nonzero. Observe that
	moreover $g$ cannot be an isomorphism, as we must have $d_1g=0$. This already
	yields $i_0\geq k_1\geq \cdots \geq k_s > i_1$. Now, assume by contradiction
	that $j_0<j_1$. Then we would have $i_0\geq k_s\geq i_1 \geq j_1 >j_0$, and since
	$0\neq d_0\in\Hom(P_{i_0},P_{j_0})$, we would deduce from the observation that
	the composition of the nonzero morphisms $d_1$ and $g$ cannot vanish, a contradiction.

	Case 3. The third case, depicted below, is dual to the previous one:
	\[\begin{tikzcd}
		P_{i_0} \ar{d}{d_0} \ar{r} & 0 \ar{d} \ar{r} & \cdots
		\ar{r} & 0 \ar{d} \ar{r} & P_{i_1} \phantom{.} \ar{d}{d_1} \\
		P_{j_0} \ar{r}{f} & P_{k_1} \ar{r} & \cdots \ar{r} & P_{k_s} \ar{r}{g} & P_{j_1}.
	\end{tikzcd}\]
\end{proof}

As a consequence of Proposition~\ref{prop:2-tilting-acyclic}, we obtain a sort of analogue of
Theorem~\ref{thm:tau-tilting-finite}, saying that over a 2-tilting acyclic
algebra, if only finitely many functorially finite torsion pairs induce derived
equivalence, then they are all the torsion pairs inducing derived equivalence.
Write $\ftors^d\Lambda:=\tors^d\Lambda\cap\ftors\Lambda$ for the set of
functorially finite torsion pairs inducing derived equivalence. We write
$\lessdot$ to denote a minimal strict inclusion.

\begin{prop}\label{prop:2-tilting-acyclic}
	Let $\Lambda$ be 2-tilting acyclic, and let $\sigma,
	\tau\in\tors^d\Lambda$ be torsion pairs giving derived equivalence. Assume
	furthermore that $\sigma$ is functorially finite. Then:
	\begin{enumerate}
		\item if $\sigma\leq \tau$, there exists $\sigma'\in\tors^d\Lambda$ functorially
			finite such that $\sigma\lessdot \sigma'\leq \tau$;
		\item if $\tau\leq \sigma$, there exists $\sigma'\in\tors^d\Lambda$ functorially
			finite such that $\tau\leq \sigma'\lessdot \sigma$;
	\end{enumerate}
\end{prop}

\begin{proof}
	We prove (1). Consider the HRS-tilt
	$\Hcal_\sigma$ of $\mod\Lambda$ at $\sigma$.
	Since $\sigma$ is functorially
	finite and it induces derived equivalence by assumption, it is associated to a
	two-term tilting complex, and $\Hcal_\sigma$ is the category of modules over its
	endomorphism ring, which is acyclic by assumption.
	Recall from
	\S\ref{subsec:hrs-tilting} that there is a poset isomorphism between the
	intervals $[\sigma,\mod \Lambda]$ of $\tors\Lambda$ and $[0,\bar\sigma]$ of
	$\tors\Hcal_\sigma$. By Lemma~\ref{lemma:acyclic} applied in
	$\Hcal_\sigma\simeq\mod(\Gamma)$, there
	exists a torsion pair $\bar\sigma'\in\tors^d\Hcal_\sigma$ such that
	$0\lessdot\bar\sigma'\leq \bar \tau$. Through the isomorphism of posets, this
	yields an element $\sigma'\in\tors^d\Lambda$ such that $\sigma\lessdot\sigma'\leq
	\tau$. This proves (1). For (2), the argument is similar, using the isomorphism between the intervals
	$[0,\sigma]\subseteq\tors\Lambda$ and
	$[\bar\sigma,\Hcal_\sigma]\subseteq\tors\Hcal_\sigma$.
\end{proof}

\begin{cor}\label{cor:dij-derived}
	Let $\Lambda$ be a 2-tilting acyclic algebra. Then, if
	$\ftors^d\Lambda$ is finite, we have that $\ftors^d\Lambda=\tors^d\Lambda$. In
	this case, $\tors^d\Lambda$ is union of maximal chains in $\tors\Lambda$.
\end{cor}

\begin{proof}
	We prove the contrapositive of the first claim. If
	$\tau\in\tors^d\Lambda\setminus\ftors^d\Lambda$,
	applying the Proposition repeatedly yields an infinite chain $\mathbf{0}\lessdot
	\sigma_1\lessdot\sigma_2\lessdot\cdots\leq \tau$, with
	$\sigma_i\in\tors^d\Lambda\cap\ftors\Lambda$. Thus $\ftors^d\Lambda$ is not
	finite.

	For the second claim, for any $\tau\in\tors^d\Lambda$, applying the Proposition
	repeatedly as we did above, we obtain a chain in $\tors^d\Lambda$:
	\[\mathbf{0}\lessdot \sigma_1\lessdot\sigma_2\lessdot\cdots\leq \tau\leq\cdots
	\lessdot\sigma^2\lessdot\sigma^1\lessdot\mathbf{1}.\]
	Since $\tors^d\Lambda=\ftors^d\Lambda$ is finite by assumption, we must have
	$\sigma_i\lessdot \tau\lessdot \sigma^j$ for some $i,j$, showing that $\tau$
	lies in a maximal chain of $\tors\Lambda$ entirely contained in
	$\tors^d\Lambda$.
\end{proof}

\begin{question}
	When $n$ is odd, the algebra $\Lambda_n$ of
	Example~\ref{ex:self-injective-two-vertices} also has the property that
	$\tors^d(\Lambda_n)(=\tors(\Lambda_n))$ is a union of maximal chains in
	$\tors(\Lambda_n)$. However, it is clearly not acyclic, let alone 2-tilting
	acyclic. Therefore, the following question remains open:

	\textit{Is it possible to characterise the algebras $\Lambda$ such that
	$\tors^d(\Lambda)$ is a union of maximal chains in $\tors(\Lambda)$?}

	Examples of such algebras are the 2-tilting acyclic algebras $\Lambda$ for which
	$\tors_d(\Lambda)$ is finite (for example, the representation-finite
	piecewise-hereditary algebras) and the $\tau$-tilting finite symmetric
	algebras.
\end{question}

%{{{ Bibliography 
\bibliographystyle{plain}
\bibliography{/home/serpav/doc/math/references.bib}
%}}}
\end{document}